\documentclass[10pt,twoside,a4paper]{article}
\usepackage[T1]{fontenc}				
\usepackage[utf8]{inputenc}				
\usepackage{color}							
\usepackage[normalem]{ulem}			
\usepackage{indentfirst}					
\usepackage{setspace}                   
\usepackage[a4paper,top=4cm,bottom=3.5cm,left=3cm,right=3cm]{geometry} 
\usepackage{amsmath, amsthm}		
\usepackage{mathtools}					
\usepackage{amssymb} 					
\usepackage{bbm}							
\usepackage{latexsym}						
\usepackage[thicklines]{cancel}			
\usepackage{bigints}							
\usepackage{marvosym}					
\usepackage{graphicx}						
\usepackage{float}							
\usepackage{tikz}							
\usepackage{xcolor}							
\usetikzlibrary{intersections,shapes.arrows}

\theoremstyle{plain}
\newtheorem{theorem}{Theorem}[section]
\newtheorem{proposition}[theorem]{Proposition}
\newtheorem{definition}[theorem]{Definition}
\newtheorem{lemma}[theorem]{Lemma}
\newtheorem{corol}[theorem]{Corollary}

\theoremstyle{remark}
\newtheorem{example}[theorem]{Example}

\title{Harmonically free group actions and equivariant \\ bifurcation  in the Yamabe problem}

\author{\normalsize{\textsc{Ana Cláudia da Silva Moreira}}\\ {\footnotesize anaclaudia77br\MVAt gmail.com.br}\\ [.5cm]Departamento de Matemática\\ Universidade de São Paulo\\ Brazil}

\date{ }

\begin{document}
	
\maketitle

\begin{abstract}
	We study multiplicity of constant scalar curvature metrics in products of a compact closed manifold and a compact manifold with boundary using equivariant bifurcation theory.
\end{abstract}

	
\section{Introduction}\label{chp:equivariante}

The classical Yamabe problem,
\begin{center}
\textit{Let $(M,g)$ be a compact Riemannian manifold (without boundary) of dimension $m \geq 3$. \\ Then, there exists a metric $\tilde{g}$, conformal to $g$, having constant scalar curvature,}
\end{center}
proposed as a conjecture by the Japanese mathematician Hidehiko Yamabe and completely proved only 29 years latter with the combined work of Yamabe \cite{Yamabe}, Trudinger \cite{Trudinger}, Aubin \cite{Aubin} and Schoen \cite{SchoenI}, \cite{SchoenII}, gave rise to a torrent of works, questions and results in Mathematics; new versions of the problem emerged and one can find it as an inspiration for the most varied works.

In 2012, Lima, Piccione and Zedda (in \cite{Lima}) proposed a slightly more general statement of an abstract bifurcation result of Smoller and Wasserman \cite{Wasserman} to determine the existence of multiple constant scalar curvature metrics on products of compact manifolds without boundary, based on a jump in the Morse index in a variational context. In fact, bifurcation techniques in the Yamabe problem have been used in many problems like in the case of collapsing Riemannian manifolds, and in the case of homogeneous metrics on spheres, see references  \cite{BettiolIII} e \cite{Bettiol0}. 

In manifolds with boundary, the Yamabe problem has several formulations:
\begin{center}
	\begin{minipage}{12cm}
		{\it Let $(M, \overline{g})$ be a compact Riemannian manifold with boundary $\partial M \neq \emptyset$, $\mbox{dim }M = m\geq 3$. It questions the existence of a metric $\tilde{g} \in [\overline{g}]$ such that
			\begin{itemize}
				\item[(i)] the scalar curvature $R_{\tilde{g}}$ is a nonzero constant and the boundary is a minimal submanifold of $M$,
				\item[(ii)] $R_{\tilde{g}} = 0$ and $\partial M$ is a submanifold of $M$ with nonzero constant mean curvature,
				\item[(iii)] the scalar curvature $R_{\tilde{g}}$ and the mean curvature of the boundary $H_{\tilde{g}}$ are both nonzero constants.
		\end{itemize}}
	\end{minipage}
\end{center}
The existence of solution for these problems was proved in the works of Cherrier \cite{Cherrier}, Escobar \cite{EscobarII}, \cite{EscobarIII}, \cite{EscobarIV}, Almaraz \cite{Almaraz}, Han \cite{Han}, Marques \cite{MarquesI}, \cite{MarquesII} and others.

In \cite{Diaz} we studied multiplicity of solutions of the the Yamabe problem, according to formulation $(i)$ above, in a manifold $(M, \overline{g}_s)$ obtained as the product of a compact Riemannian manifold $(M_1, g^{(1)})$ (without boundary) with constant scalar curvature and a compact Riemannian manifold $(M_2, \overline{g}^{(2)})$ with constant scalar curvature and minimal boundary, where $M = M_1 \times M_2$, $\partial M = M_1 \times \partial M_2 \neq \emptyset$, and $\overline{g}_s = g^{(1)} \oplus s \overline{g}^{(2)}$ is a family of metrics of constant scalar curvature and vanishing mean curvature of the boundary, for all $s \in (0, \infty)$.

The Yamabe problem, in formulation $(i)$, is mathematically translated in an elliptic non-linear P.D.E. eigenvalue problem. Consider a conformal metric $\tilde{g} = u^{\frac{4}{m-2}} \overline{g}$, where $u$ is a positive smooth function on $M$. If $R_{\tilde{g}}$ and $R_{\overline{g}}$ denote the scalar curvature with respect to $\tilde{g}$ and $\overline{g}$ respectively, then the Laplacian relate the two of them by means of the Neumann problem:
\begin{equation}\label{eq:EDP}
\left\{
\begin{array}{cccl}
\displaystyle\frac{4(m-1)}{m-2} \Delta_{\overline{g}} u + R_{\overline{g}}u - K u ^{\frac{m+2}{m-2}}& =& 0, & \mbox{em } M \\
(m-1) \partial_{\eta_{\overline{g}}} u + \displaystyle\frac{m-2}{2} H_{\overline{g}} u &=&  0, & \mbox{em } \partial{M}
\end{array}
\right. 
\end{equation}
where
$K = \displaystyle\frac{(m-2)}{4(m-1)}R_{\tilde{g}}$
is constant. The problem has a variational structure, that is, $u>0$ is a solution for \eqref{eq:EDP} if and only if it is a critical point for the Hilbert-Einstein functional
$$
\begin{array}{cc}
F(g) = \displaystyle\int_M R_{g} \ \omega_{g},
\end{array}
$$
restrict to the normalized conformal class, with volume $1$, of $\overline{g}$:
$$[\overline{g}]^{0}_{1} = \left\{ \phi \overline{g} : \phi \in \mathcal{C}_{+}^{k, \alpha}(M), \ \partial_{\eta_{\overline{g}}} \phi = 0, \ \displaystyle\int_{M} \phi^{\frac{m}{2}} \ \omega_{\overline{g}} = 1 \right\}.$$

In \cite{Diaz}, we used a rigidity result, adapted from \cite{Uniqueness} to the boundary case, to establish a degeneracy condition on the critical points of $F$, this condition is essential because the bifurcation instants are found among the degeneracy instants. We showed that there is a countable subset $S \subset (0, \infty)$ of instants such that, for all instants $s \in (0, \infty) \backslash S$, the family $\{\overline{g}_s\}_{s>0}$ is rigid, that is, save homotheties, there exists local uniqueness of solution of the Yamabe problem in the normalized conformal classes of the corresponding metrics $\overline{g}_s$. On the other hand, proposing a natural extension of bifurcation theorem of \cite{Lima}, we proved that, \emph{except for a finite number of instants}, all  the others $s_* \in S$ are bifurcation instants, i.e. there is a sequence of solutions $\overline{g}_{s_n}$ of the Yamabe problem converging to $\overline{g}_{s_*}$, each of which represents a second solution of the Yamabe problem in the correspondent $\overline{g}_s$-normalized conformal class. These were the same results obtained by Lima, Piccione and Zedda in \cite{Lima}, when studying multiplicity of solutions of the Yamabe problem in product manifolds (without boundary).

The motivation of this work is to classify the \emph{finite number of instants} that remained undecidable in \cite{Diaz}, called \emph{neutral instants}. In this article, we introduce the notion of \emph{harmonically free action} of a Lie group $G$, necessary to establish bifurcation in cases where it can not be detected a jump in the Morse index.

\begin{definition}
	We say that the isometric action of a Lie group $G$ on a Riemannian manifold $(N,h)$ is harmonically free if, given an arbitrary family of eigenspaces of the Laplacian $\Delta_h$, pairwise distinct,
	$$V_1, V_2, \ldots, V_r, V_{r+1}, \ldots, V_{r+s}, \ \ \ r, s \geq 1,$$
	and integers $n_{\ell} \geq 0$, with $\ell = 1, \ldots, r+s$, not all zero, the (anti-)representations
	$$\bigoplus_{\ell = 1}^{r}  n_{\ell} \cdot \pi_{\ell} \ \ \mbox{ e }  \ \ \bigoplus_{\ell=r+1}^{r+s} n_{\ell} \cdot \pi_{\ell},$$
	are non-equivalents. Here $\pi_{\ell}$ denotes the natural representation of $G$ on $V_\ell$ (obtained by right compositions with the isometries of the $G$-action).
\end{definition}

 Harmonically free actions are easily found among compact symmetric spaces of rank one: 
 \begin{itemize}
 	\item[(i.)] the real projective space, with $M = \mathbb{R}P^n$, $G=O(n+1)$,
 	\item[(ii.)] the complex projective space, with $M=\mathbb{C}P^n$, $G=U(n+1)$, 
 	\item[(iii.)] the quaternionic projective space, with $M=\mathbb{H}P^n$, $G = Sp(n+1)$, 
 	\item[(iv.)] the Cayley plane, with $M=\mathbb{P}^2(Ca)$, $G=F_4$. 
 \end{itemize}

In addition to this notion of harmonically free action, the facts that the eigenspaces of the Laplacian are invariant by the isometric action of $G$ (Proposition \ref{prop:autesp}) and that its isometric action determines (anti-)representations of $G$ in the eigenspaces of the Laplacian (Proposition \ref{prop:repres}) are crucial for the development of our main result. Consider the following setup. Let $(M_1, g^{(1)})$ be a compact Riemannian manifold with positive constant scalar curvature and let $(M_2, \overline{g}^{(2)})$ be a compact Riemannian manifold with boundary $\partial M \neq \emptyset$, positive constant scalar curvature and vanishing mean curvature of the boundary. Assume the pair $(g^{(1)}, \overline{g}^{(2)})$ is non-degenerate. Consider the product manifold $M = M^{(1)} \times M^{(2)}$, of dimension $m \geq 3$ and boundary $\partial M = M^{(1)} \times \partial M^{(2)} \neq \emptyset$. Let $\overline{g}_s = g^{(1)} \oplus s \overline{g}^{(2)}$, with $s>0$, be a family of metrics in $M$. Let $G$ be a connected Lie group action by isometries on $M_1$.

\begin{theorem}
	If the action of $G$ on $M_1$ is harmonically free, then every degenerate instant for the family of operators $\left\{\mathcal{J}_s \right\}_{s>0}$ is an instant of bifurcation for the family $\{\overline{g}_s\}_{s>0}$ of solutions of the Yamabe problem in the compact manifold $M$ with minimal boundary.
\end{theorem}

This paper is organized as follows. In Section \ref{sec:general} we recall from \cite[Section 2]{Diaz}, basically it defines the manifolds, normalized conformal classes, functional, variations and critical points. In particular, in Subsection \ref{sec:iso_rep} we prove some simple, but important results about isometric actions and representations in our problem context. In Section \ref{sec:rigid}, we remember the rigidity result of \cite{Uniqueness}, naturally extended for manifolds with boundary in \cite{Diaz} and present a natural extension of the equivariant bifurcation theorem (\cite[Theorem~3.4]{Lima}) to the case of manifolds with boundary (Theorem \ref{thm:bifeq}). In Section \ref{sec:harm} we discuss the harmonically free action concept and, finally, in Section \ref{sec:main} we present our main result, Theorem \ref{thm:maineq}.

{\bf Acknowledgment:} The author would like to thank professor Paolo Piccione for his reading of this article and precious suggestions.

\section{General Settings}\label{sec:general}

\subsection{Manifolds and Conformal Classes}\label{sec:variedades}

We summarize here many important facts about conformal classes. The proof for all these results can be found in \cite{Diaz}, however, we repeat the most important ones here for the reader's convenience. 

Let $(M, \overline{g})$ be a $m$-dimensional oriented compact Riemannian manifold with boundary $\partial M \neq \emptyset$, $m \geq 3$, then the space $\Gamma^{k} (T^*M \otimes T^*M)$ of $\mathcal{C}^{k}$-sections of the vector bundle $T^*M \otimes T^*M$ of symmetric $(0,2)$-tensors of class $\mathcal{C}^k$ on $M$ is a Banach space with the norm
$$\left\| \tau \right\|_{\mathcal{C}^{k}} = \max_{j=0, \ldots, k} \left[ \max_{p \in M} \left\| \overline{\nabla}^{(j)} \tau(p) \right\|_{\overline{g}} \right],$$
and therefore, it is a Banach manifold.

We denote by $\mathcal{M}^{k, \alpha}(M)$ the set of all Riemannian metrics of class $\mathcal{C}^{k, \alpha}$ on $M$, with $k \geq 3$ and $\alpha \in (0,1]$. The set $\mathcal{M}^{k, \alpha}(M)$ is an open cone of $\Gamma^{k, \alpha} (T^*M \otimes T^*M)$, so this is a Banach manifold itself and its tangent space is $T_{\overline{g}}\mathcal{M}^{k, \alpha}(M) = \Gamma^{k, \alpha} (T^*M \otimes T^*M)$, for a metric $\overline{g} \in \mathcal{M}^{k, \alpha}(M)$.

Denote by $\mathcal{C}_{+}^{k, \alpha}(M)$ the open subset of the Banach space $\mathcal{C}^{k, \alpha}(M)$ consisting of positive functions. Now, we define the $\mathcal{C}^{k, \alpha}$-conformal class of $\overline{g}$ by
$$[\overline{g}] \colon = \{ \phi \overline{g} : \phi \in \mathcal{C}^{k, \alpha}_{+}(M) \},$$
which is an embedded submanifold of $\mathcal{M}^{k, \alpha}(M)$. In particular, $[\overline{g}]$ is a Banach manifold with differential structure induced by $\mathcal{C}^{k, \alpha}(M)$ and its tangent space, $T_{\overline{g}}[\overline{g}] = \{\psi \overline{g} : \psi \in \mathcal{C}^{k, \alpha}(M) \},$ can be identified with $\mathcal{C}^{k, \alpha}(M)$.

Let $\omega_{\overline{g}}$ be the volume form on $M$, with respect to $\overline{g}$, and $\sigma_{\overline{g}}$ the volume form induced on $\partial M$. The volume function defined by
$$\mathcal{V}(\overline{g}) \colon = \displaystyle\int_{M} \ \omega_{\overline{g}}$$
is smooth on $\mathcal{M}^{k, \alpha}(M)$ and its differential is given by 
\begin{equation}\label{eq:derVol}
	(d\mathcal{V})_{\overline{g}}(h) = \displaystyle\frac{1}{2} \displaystyle\int_{M} \mbox{tr}_{\overline{g}}(h) \ \omega_{\overline{g}},
\end{equation}
for each $\overline{g} \in \mathcal{M}^{k, \alpha}(M)$ and $h \in T_{\overline{g}}\mathcal{M}^{k, \alpha}(M)$.

The subset of unitary volume metrics in $\mathcal{M}^{k, \alpha}(M)$,
$$\mathcal{M}^{k, \alpha}(M)_{1} \colon = \big\{ \overline{g} \in \mathcal{M}^{k, \alpha}(M) : \mathcal{V}(\overline{g}) = 1 \big\},$$
is a smooth embedded submanifold of $\mathcal{M}^{k, \alpha}(M)$ (see \cite{Diaz}). Observe that if $\overline{g} \in \mathcal{M}^{k, \alpha}(M)_{1}$, the conformal metric $\phi \overline{g}$, for some $\phi \in \mathcal{C}_{+}^{k, \alpha}(M)$, is not in $\mathcal{M}^{k, \alpha}(M)_{1}$, in general. Indeed,
$$\mathcal{V}(\phi \overline{g}) = \displaystyle\int_{M} \omega_{\phi \overline{g}} = \displaystyle\int_{M} \phi^{\frac{m}{2}} \ \omega_{\overline{g}}.$$
To avoid homotheties, for each $\overline{g} \in \mathcal{M}^{k, \alpha}(M)_{1}$, we define
$$[\overline{g}]_{1} = \left\{ \phi \overline{g} : \phi \in \mathcal{C}_{+}^{k, \alpha}(M), \displaystyle\int_{M} \phi^{\frac{m}{2}} \ \omega_{\overline{g}} = 1 \right\},$$
which is an embedded submanifold of $[\overline{g}]$ and an embedded submanifold of $\mathcal{M}^{k, \alpha}(M)_{1}$ (see \cite{Uniqueness}), with tangent space
$$T_{\overline{g}}[\overline{g}]_{1} = \left\{ \psi \overline{g} : \psi \in \mathcal{C}^{k, \alpha}(M), \ \displaystyle\int_{M} \psi \ \omega_{\overline{g}} = 0 \right\}.$$
Note that $[\overline{g}]_{1} = \mathcal{M}^{k, \alpha}(M)_{1} \cap [\overline{g}]$; in fact, $\mathcal{M}^{k, \alpha}(M)_{1}$ is transverse to $[\overline{g}]$, so $[\overline{g}]_{1}$ is an embedded submanifold of  $\mathcal{M}^{k, \alpha}(M)$ (see \cite{Uniqueness} for details).

Now, we define the $\mathcal{C}^{k, \alpha}$-normalized conformal class, the manifold that is specially interesting for our problem
$$[\overline{g}]^{0} = \{\tilde{g} \in [\overline{g}] : H_{\tilde{g}} = 0 \},$$
where, $H_{\tilde{g}}$ is the mean curvature of the boundary with respect to the metric $\tilde{g}$.
By a result of Escobar \cite{EscobarII} \textit{there is at least one metric with vanishing mean curvature in each conformal class}, so given a conformal class $[\overline{g}]$ we can assume that $H_{\overline{g}} = 0$.

\begin{proposition}
	The $\mathcal{C}^{k, \alpha}$-normalized conformal class of $\overline{g}$ can be identified with
	$$\mathcal{C}^{k, \alpha}_{+}(M)^{0} = \left\{\phi \in \mathcal{C}^{k, \alpha}_{+}(M) : \partial_{\eta_{\overline{g}}} \phi = 0 \mbox{ on } \partial M \right\},$$
	which is a closed subset of $\mathcal{C}^{k,\alpha}(M)$.
\end{proposition}
\begin{proof}
	Consider $\tilde{g} = \phi^{\frac{4}{m-2}} \overline{g} \in [\overline{g}]^{0}$. The mean curvature with respect to $\tilde{g}$ is given by 
	$$H_{\tilde{g}} = \phi^{-\frac{m}{m-2}} \left( H_{\overline{g}} \phi + \displaystyle\frac{2(m-1)}{(m-2)} \partial_{\eta_{\overline{g}}} \phi \right);$$
	as $H_{\overline{g}} = 0$ and $\phi > 0$, it follows that $H_{\tilde{g}} = 0$ if and only if $\partial_{\eta_{\overline{g}}} \phi = 0$.
	Now, if $(\phi_n)_n$ is a sequence of functions in $\mathcal{C}^{k, \alpha}_{+}(M)^{0}$ converging to a function $\phi$ with respect to the norm $\mathcal{C}^{k, \alpha}$, then for all $n$ greater than a certain $n_0$, we have
	\begin{align*}
	\| \phi_n - \phi \|_{\mathcal{C}^{k, \alpha}} < \varepsilon \ \
	&\Rightarrow \ \  \| \phi_n - \phi \|_{\mathcal{C}^k} < \varepsilon \ \ \Rightarrow \ \ \max_{0 \leq |a| \leq k} \left[ \max_{p \in M} \left| D^{a}  (\phi_n - \phi) (p) \right| \right] < \varepsilon \\
	& \Rightarrow \ \ \max_{p \in M} \left| \partial_{\eta_{\overline{g}}}\phi_n(p) - \partial_{\eta_{\overline{g}}} \phi(p) \right| < \varepsilon \ \ \Rightarrow \ \  \| \partial_{\eta_{\overline{g}}}\phi_n - \partial_{\eta_{\overline{g}}} \phi  \|_{\mathcal{C}^0} < \varepsilon,
	\end{align*}
	as $\partial_{\eta_{\overline{g}}}\phi_n = 0$ for all $n$, it follows that $\partial_{\eta_{\overline{g}}} \phi = 0$, so $\phi \in \mathcal{C}^{k, \alpha}_{+}(M)^{0}$, which concludes the proof.
\end{proof}

To prove the next proposition we need the following Lemma which is proved in \cite[Proposition~2.5]{Diaz} and is an elementary version of a more general result that can be found in \cite[Theorem~4, Ch.\ IV]{Stein}.

\begin{lemma}\label{lemma}
	There exists a continuous linear map
	$$\mathcal{F} : \mathcal{C}^{k,\alpha}(\partial M) \longrightarrow \mathcal{C}^{k+1,\alpha}(M)$$
	such that, for $\xi \in \mathcal{C}^{k,\alpha}(\partial M)$ the following properties are satisfied:
	\begin{enumerate}
		\item [(a.)] $\mathcal{F}(\xi)$ vanishes on $\partial M$;
		\item [(b.)] $\partial_{\eta} \mathcal{F}(\xi) = \xi$.
	\end{enumerate}
\end{lemma}

\begin{proposition}
	The $\mathcal{C}^{k, \alpha}$-normalized conformal class, $[\overline{g}]^{0}$, is an embedded submanifold of $[\overline{g}]$.
\end{proposition}
\begin{proof}
Given $\overline{g} \in \mathcal{M}^{k, \alpha}(M)$, let $\eta_{\overline{g}}$ be the unitary (inward) vector field normal to the boundary $\partial M$. Define
$$
\begin{array}{rcl}
\mathcal{N}_{\overline{g}} : [\overline{g}] & \longrightarrow & \mathcal{C}^{k-1, \alpha}(\partial M) \\
\phi \overline{g} & \longmapsto & \partial_{{\eta}_{\overline{g}}} \phi.
\end{array}
$$
So, $\mathcal{N}_{\overline{g}}^{-1}\left( \left\{ 0 \right\} \right) = [\overline{g}]^{0}$ and the differential $(\mbox{d}\mathcal{N}_{\overline{g}})_{\phi \overline{g}} : \mathcal{C}^{k, \alpha}(M) \longrightarrow \mathcal{C}^{k-1, \alpha}(\partial M)$ is given by $(\mbox{d}\mathcal{N}_{\overline{g}})_{\phi \overline{g}} (\psi) = \partial_{\eta_{\overline{g}}} \psi$, for all $\phi \overline{g} \in [\overline{g}]$ and $\psi \in \mathcal{C}^{k, \alpha}(M)$. Now, by Lemma \ref{lemma}, $\left( \mbox{d} \mathcal{N}_{\overline{g}} \right)_{\phi \overline{g}}$ admits a bounded right-inverse, for all $\phi\overline{g}\in [\overline{g}]$. Therefore, the differential is surjective and its kernel,
$$\mbox{ker } (\mbox{d}\mathcal{N}_{\overline{g}})_{\phi \overline{g}} = \left\{ \psi \in \mathcal{C}^{k, \alpha}(M) : \partial_{\eta_{\overline{g}}} \psi = 0 \right\},$$
has a closed complement in $\mathcal{C}^{k, \alpha}(M)$. It follows that $[\overline{g}]^{0}$ is an embedded submanifold of $[\overline{g}]$.
\end{proof}

Finally, we define the $\mathcal{C}^{k, \alpha}$-normalized conformal class consisting of metrics of volume one
$$[\overline{g}]^{0}_{1} = \left\{ \phi \overline{g} : \phi \in \mathcal{C}_{+}^{k, \alpha}(M), \ \partial_{\eta_{\overline{g}}} \phi = 0, \ \displaystyle\int_{M} \phi^{\frac{m}{2}} \ \omega_{\overline{g}} = 1 \right\},$$
that is an embedded submanifold of $\mathcal{M}^{k, \alpha}(M)_{1}$ and an embedded submanifold of $[\overline{g}]$. This conformal class can be expressed as $[\overline{g}]^{0}_{1} = [\overline{g}]^{0} \cap \mathcal{M}^{k, \alpha}(M)_{1}$, and the correspondent tangent space is identified with
$$T_{\overline{g}}[\overline{g}]^{0}_{1} = \left\{ \psi \in \mathcal{C}^{k, \alpha}(M)^{0} : \displaystyle\int_{M} \psi \ \omega_{\overline{g}} = 0 \right\}.$$

\subsection{The Hilbert-Einstein Functional}\label{sec:HEfunc}

The Hilbert-Einstein functional compose the variational model for the Yamabe problem. In this section we present some well-known facts about the Hilbert-Einstein functionl that are proved in \cite{Diaz}.

The Hilbert-Einstein functional $F: \mathcal{M}^{k, \alpha}(M) \longrightarrow \mathbb{R}$ given by
\begin{equation} \label{functional}
F(\overline{g}) = \displaystyle\int_{M}{R_{\overline{g}} \ \omega_{\overline{g}}},
\end{equation}
is a smooth functional in $\mathcal{M}^{k, \alpha}(M)$ and in $[\overline{g}]$. The first variation of $F$ is given by 
$$\delta F(\overline{g})h = - \displaystyle\int_{M} \left\langle Ric_{\overline{g}} - \displaystyle\frac{1}{2}R_{\overline{g}} {\overline{g}},h\right\rangle_{\overline{g}} \omega_{\overline{g}} - 2 \displaystyle\int_{\partial M} \left( \delta H_{\overline{g}} + \frac{1}{2} \left\langle I\!I_{\overline{g}},h \right\rangle\right) \sigma_{\overline{g}},$$
for all $h \in \Gamma^{k, \alpha}(T^*M \otimes T^*M)$.

For compact Riemannian manifolds (without boundary), we know that the critical points of $F$, restricted to $\mathcal{M}^k(M)_{1}$, are the Einstein metrics of volume $1$ in $M$ and, if $F$ is restricted to the $\mathcal{C}^{k, \alpha}$-conformal classe of volume $1$, $[\overline{g}]_{1}$, then the critical points are those metrics conformal to $\overline{g}$, which have unitary volume and constant scalar curvature. 

We are interested in the critical points of $F$ restricted to the $\mathcal{C}^{k, \alpha}$-normalized conformal class of volume $1$, denoted by $[\overline{g}]^0_1$, once we are dealing with manifolds with boundary. The critical points of $F$ in $[\overline{g}]^{0}_{1}$ are metrics which are conformal to $\overline{g}$, have unitary volume, constant scalar curvature and vanishing mean curvature.

If $\overline{g} \in \mathcal{M}^{k, \alpha}(M)_{1}$ is a critical point of $F$ in $[\overline{g}]^{0}_{1}$, then (see \cite{Koiso}) the second variation of $F$ is given by the quadratic form
$$\delta^2 {F}({\overline{g}})(\psi)  = \displaystyle\frac{(m-2)}{2} \displaystyle\int_{M} \left( (m-1) \Delta_{\overline{g}} \psi - R_{\overline{g}} \psi \right) \psi \ \omega_{\overline{g}},
$$
where $\psi \in \mathcal{C}^{k, \alpha}(M)^0$ and has null integral, and $\overline{g}$ is non-degenerate if $R_{\overline{g}} = 0$ or $\frac{R_{\overline{g}}}{m-1}$ is not an eigenvalue of $\Delta_{\overline{g}}$ with Neumann boundary condition.

Indeed, note that $(m-1) \Delta_{g} - R_{g}$, as an operator from $\mathcal{C}^{k,\alpha}(M)$ in $\mathcal{C}^{k-2, \alpha}$, is Fredholm of index zero, and is precisely for this reason that we are considering these Hölder spaces. Note that this operator maps the subespace
$$\left\{ \psi \in \mathcal{C}^{k, \alpha}(M) : \partial_{\eta_{\overline{g}}} \psi = 0, \displaystyle\int_{M} \psi = 0 \right\}$$
into the subespace of $\mathcal{C}^{k-2, \alpha}$ which consists of functions with null integral. It follows that
$$(m-1) \Delta_{\overline{g}} - R_{\overline{g}} :  T_{\overline{g}}[\overline{g}]_{1} \longrightarrow \mathcal{C}^{k-2, \alpha}(M)$$
is Fredholm of index zero. So, the quadratic form $\delta^2 F (\overline{g})(\psi,\psi)$ is non-degenerate if and only if $\ker  \left( (m-1) \Delta_{g} - R_{g} \right) = \{0\}$, that is, if and only if $\frac{R_{\overline{g}}}{m-1} = 0$ or $\frac{R_{\overline{g}}}{m-1}$ is not a solution for the Neumann problem
$$
\left\{
\begin{array}{lcl}
\Delta_{\overline{g}} \psi & = & \rho \psi\\
\partial_{\eta_{\overline{g}}} \psi & = & 0
\end{array}.
\right.
$$

\subsection{Isometric Actions and Representations}\label{sec:iso_rep}

Let $(M,g)$ be a compact Riemannian manifold (with or without boundary). In the space of $\mathcal{C}^k(M)$-functions, with $k \geq 2$, we define the Laplacian operator of $M$ with respect to the Riemannian metric $g$ by
$$\Delta_g = - \mbox{div}_g \nabla_g.$$
The following eigenvalue problems are classics in compact manifolds:
\begin{itemize}
	\item[(1)] Closed problem: $\partial M = \emptyset$,
	$$\Delta f = \rho f, \mbox{ em } M,$$
	\item[(2)] Dirichlet problem: $\partial M \neq \emptyset$,
	$$
	\left\{
	\begin{array}{rcl}
	\Delta f & = & \rho f, \mbox{ em } M, \\
	f & = & 0, \mbox{ em } \partial M
	\end{array}
	\right.
	$$
	\item[(3)] Neumann problem: $\partial M \neq \emptyset$,
	$$
	\left\{
	\begin{array}{rcl}
	\Delta f & = & \rho f, \mbox{ em } M, \\
	\partial_{\eta} f & = & 0, \mbox{ em } \partial M
	\end{array}
	\right.
	$$
	where $\eta$ is the unitary field normal to the boundary.
\end{itemize}
The Laplacian operator $\Delta_g$ is symmetric with respect to the inner product $L^2(M)$ in the space $\mathcal{C}^k(M)$, $k \geq 2$, for any of the boundary conditions given in the above problems.  It is well-known (\cite{Berard}, p. 53) that, for each of these eigenvalue problems, the Laplacian has an infinity sequence of non-negative eigenvalues $\rho$, of finite geometric multiplicity
$$0 \leq \rho_0 < \rho_1 < \ldots < \rho_k < \ldots,$$
that is, the eigenspace $V_{\rho}$ of $\Delta_g$ associated to the eigenvalue $\rho$, solution of one of the previous problems, has finite dimension.

Consider, now, a Lie group $G$, acting by isometries on $(M,g)$:
$$\mu : G \times M \longrightarrow M.$$
The next result ensures the invariance of the eigenspaces of the Laplacian of $M$ by the action of $G$.

\begin{proposition} \label{prop:autesp}
	If $G$ is a Lie group, acting by isometries on $(M,g)$ and $V_{\rho}$ is an eigenspace of the Laplacian $\Delta_g$ associated to the eigenvalue $\rho$, solution of any of the eigenvalue problems above, then $V_{\rho}$ is invariant by the action of $G$. 
\end{proposition}
\begin{proof}
	Consider $\xi \in G$ and $f \in V_{\rho} \subset \mathcal{C}^{k}(M)$, $k \geq 2$, then
	\begin{itemize}
		\item[(a.)] $f \circ \xi$ is a function of class $\mathcal{C}^{k}$ in $M$,
		\item [(b.)] $\Delta_{g}(f \circ \xi)(x) = (\Delta_{g} f) (\xi(x)) = \rho f (\xi(x)) =  \rho (f \circ \xi) (x), \forall x \in M$.
	\end{itemize} 
	Further,
	\begin{itemize}
		\item if $\rho$ is solution of the closed problem and $f$ is an eigenfunction associated to $\rho$, it follows from items $(a.)$ and $(b.)$ that $f \circ \xi \in V_{\rho}$;
		\item if $\rho$ is solution of the Dirichlet problem, we have
		$$\left. (f \circ \xi)\right|_{\partial M} = f\!\left.\left(\xi\right) \right|_{\partial M} = 0,$$
		that is, $f \circ \xi \in V_{\rho}$;
		\item if $\rho$ is solution of the Neumann problem, then
		$$\partial_{\eta}(f \circ \xi) = (\partial_{\eta}f)\circ \xi \cdot \partial_{\eta} \xi = 0,$$
		so $f \circ \xi \in V_{\rho}$,
	\end{itemize}
	which concludes the proof.
\end{proof}
Generally, if $F:(M,g) \longrightarrow (N,h)$ is an isometry, then $(M,g)$ and $(N,h)$ have the same spectrum and if $f$ is an eigenfuction for $\Delta_h$, then $f \circ F$ is an eigenfunction for $\Delta_g$ associated to the same eigenvalue.

Now, we see that the isometric action of a Lie group $G$, on a compact manifold $(M,g)$, determines (anti-)representations of $G$ in the eigenspaces $V_{\rho}$ of the Laplacian $\Delta_g$, via pull-back, 
$$f \mapsto \xi^*(f) = f \circ \xi,$$
with $\xi \in G$ e $f \in V_{\rho}$.

\begin{proposition}\label{prop:repres}
	For each eigenvalue $\rho$, solution of one of the mentioned problems, \linebreak $\pi_{\rho} : G \rightarrow GL(V_{\rho}),$
	defined by the linear application 
	$$
	\begin{array}{rclcl}
	\pi_{\rho}(\xi) & : & V_{\rho} & \longrightarrow & V_{\rho} \\
	& & f & \mapsto & f \circ \xi
	\end{array}
	$$
	is a (anti-)representation of $G$ in the eigenspace $V_{\rho}$.
\end{proposition}
\begin{proof}
	We saw that the eigenspaces of the Laplacian are invariants by the action of $G$, remains to show that $\pi_{\rho}$ is an (anti-)homomorphism of Lie groups. Indeed, $\pi_{\rho}$ is smooth and for all $\xi, \zeta \in G$, $f \in V_{\rho}$, we have
	\begin{align*}
	\pi_{\rho}(\xi \zeta)f 
	& = f \circ (\xi \zeta)  = f \circ (\xi \circ \zeta )  = ( f \circ \xi ) \circ \zeta \\
	& = \pi_{\rho}(\zeta)(f \circ \xi) = \pi_{\rho}(\zeta) \left(\pi_{\rho}(\xi) f \right)  = \pi_{\rho}(\zeta) \circ \pi_{\rho}(\xi)f,
	\end{align*}
	and the result follows.
\end{proof}

\subsection{The Laplace Operator}\label{sec:laplacian}

The Laplacian on a Riemannian Manifold $(M,g)$ is the linear operator
$$\Delta_g : \mathcal{C}^k(M) \longrightarrow \mathcal{C}^{k-2}(M),$$
where $k \geq 2$, defined by
$$\Delta_g f = - \mbox{div}_g \nabla_g f. $$
If $M$ is a compact Riemannian manifold with minimal boundary, the Laplace operator is symmetric with respect to $L^2(M)$-product, in the space of smooth functions defined on $M$, which satisfies the Neumann condition $\partial_{\eta} ( \ \cdot \ ) = 0$. Also the laplacian is symmetric on the space of functions satisfying the Dirichlet condition and on the space of sufficiently regular functions on $M$, when $\partial M = \emptyset$.

Besides that, there are other well-known general facts (\cite{Berard}, pag. 53) about \emph{the spectrum of the laplacian} on a compact manifold, that we exhibited below. 
\begin{theorem}\label{thm:laplaciano}
	Let $(M,g)$ be a compact manifold with boundary $\partial M$, and consider the eigenvalue problem $\Delta_g \phi = \lambda \phi$, $\partial M \neq \emptyset$, with Neumann boundary conditions or with Dirichlet boundary condition. Then,
	\begin{itemize}
		\item[(i.)] the set of eigenvalues consist of an infinity sequence
		$$0 \leq \lambda_0 < \lambda_1 < \ldots < \lambda_k < \ldots,$$
		where $\lambda_0 = 0$ is not an eigenvalue of the Dirichlet problem;
		\item[(ii.)] each eigenvalue has finite multiplicity and the eigenspaces associated to different eigenvalues are $L^2(M)$-orthogonal;
		\item[(iii.)] the direct sum of the eigenspaces $V_{\lambda_i}$ is a dense subset of $L^2(M)$ with the norm $L^2$. Moreover, each eigenfunction is smooth.
	\end{itemize}
\end{theorem}

Now, consider the product manifold $M = M_1 \times M_2$, where $(M_1, g^{(1)})$ is closed and $(M_2, \overline{g}^{(2)})$ has non-empty boundary, with the metric $g = g^{(1)} \oplus \overline{g}^{(2)}$. There exist orthonormal basis, $\{f^{(1)}_{\alpha}\}_{\alpha}$ of $L^2(M_1)$ and $\{f^{(2)}_{\beta}\}_{\beta}$ of $L^2(M_2)$, consisting of eigenfunctions of $\Delta_{g^{(1)}}$ and of $\Delta_{\overline{g}^{(2)}}$, respectively. We claim that $\mathcal{B} = \{ f^{(1)}_{\alpha} \otimes f^{(2)}_{\beta} \}_{\alpha, \beta}$ is an orthonormal basis of eigenfunctions of $\Delta_g$ for $L^2(M)$. Indeed, if $f^{(1)}, f^{(2)}$ are eigenfunctions of $\Delta_{g^{(1)}}, \Delta_{\overline{g}^{(2)}}$, associated to the eigenvalues $\rho^{(1)}, \rho^{(2)}$, respectively, then $f^{(1)} \otimes f^{(2)} $ is an eigenfunction of $\Delta_g$ associated to the eigenvalue $\rho^{(1)} + \rho^{(2)}$. Besides that, the set
$$\mathcal{A} = \left\{ \displaystyle\sum_{\alpha, \beta} k_{\alpha, \beta} \left(f^{(1)}_{\alpha} \otimes f^{(2)}_{\beta}\right) , \mbox{ for all } k_{\alpha, \beta} \in \mathbb{R} \right\},$$
of finite linear combinations of functions of the form $f^{(1)}_{\alpha} \otimes f^{(2)}_{\beta}$, is dense in $L^2(M)$ .

As
$$\Delta_g \left(f^{(1)} \otimes f^{(2)} \right) = \left( \Delta_{g^{(1)}} f^{(1)}\right) \otimes f^{(2)}  +  f^{(1)} \otimes  \Delta_{\overline{g}^{(2)}} f^{(2)}, $$
for all $f^{(1)} \in L^2(M_1)$, $f^{(2)} \in L_2(M_2)$, we write
$\Delta_g = \Delta_{g^{(1)}} \otimes I_2  +  I_1 \otimes  \Delta_{\overline{g}^{(2)}},$
where $I_i$ is the identity operator on $L^2(M_i)$, $i=1,2$.

Furthermore, it is important to see that all the eigenvalues of the Laplacian $\Delta_g$ are of the form $\rho^{(1)} + \rho^{(2)}$. In fact, let $f \neq 0$ be an eigenfunction of $\Delta_g$ associated to the eigenvalue $\lambda$. Consider the standard inner product on the space $L^2(M)$, then for each $f_{\alpha}^{(1)} \otimes f_{\beta}^{(2)} \in \mathcal{B}$, we have
$$
\begin{aligned}
\lambda \left\langle f , f^{(1)}_{\alpha} \otimes f^{(2)}_{\beta} \right\rangle
&= \left\langle \lambda f , f^{(1)}_{\alpha} \otimes f^{(2)}_{\beta} \right\rangle \\
&= \left\langle \Delta_g f, f^{(1)}_{\alpha} \otimes f^{(2)}_{\beta} \right\rangle \\
&= \left\langle f, \Delta_g  \left( f^{(1)}_{\alpha} \otimes f^{(2)}_{\beta} \right) \right\rangle \\
&= \left\langle f, \left( \rho_{\alpha}^{(1)} + \rho_{\beta}^{(2)} \right) f^{(1)}_{\alpha} \otimes f^{(2)}_{\beta} \right\rangle \\
&= \left( \rho_{\alpha}^{(1)} + \rho_{\beta}^{(2)} \right) \left\langle f, f^{(1)}_{\alpha} \otimes f^{(2)}_{\beta} \right\rangle \\
\end{aligned}
$$
put $\rho_{\alpha, \beta} = \rho_{\alpha}^{(1)} + \rho_{\beta}^{(2)} $, then
$\lambda - \rho_{\alpha,\beta} \langle f , f^{(1)}_{\alpha} \otimes f^{(2)}_{\beta} \rangle = 0,$
for all $\alpha, \beta$, that is, if we suppose that $\lambda \neq \rho_{\alpha,\beta}$, then $f$ is orthogonal to the space $\mathcal{A}$, but the only element which is orthogonal to a dense subspace is zero, hence $f = 0$, which is a contradiction.

\section{Rigidity and equivariant bifurcation}\label{sec:rigid}

\subsubsection{Local rigidity}

In this section we remember some rigidity results obtained in \cite{Uniqueness} and written in a slightly different way in \cite{Diaz}. We refer to \cite[Proposition~3 and Corollary~4]{Uniqueness} for details.

\begin{definition}
	{\rm Let $\overline{g} \in \mathcal{M}^{k, \alpha}(M)_1$ with constant scalar curvature $R_{\overline{g}}$ in $M$. We say that $\overline{g}$ is \textbf{nondegenerate} if either $R_{\overline{g}} = 0$ or if $\frac{R_{\overline{g}}}{(m-1)}$ is not an eigenvalue of $\Delta_{\overline{g}}$, with the Neumann boundary condition $\partial_{{\eta}_{\overline{g}}} f = 0$. In other words, $\frac{R_{\overline{g}}}{(m-1)}$ is not a solution of the eigenvalue problem
		\begin{equation}
		\left\{
		\begin{array}{rcl}
		\Delta_{\overline{g}} f & = & \lambda f, \mbox{ on $M$} \\
		\partial_{{\eta}_{\overline{g}}} f & = & 0, \mbox{ on $\partial M$}.
		\end{array}
		\right.
		\label{EP}
		\end{equation}}
\end{definition}

\begin{proposition}\label{thm:rigidity}
	Let $\overline{g}_{*} \in \mathcal{M}^{k, \alpha}(M)_{1}$ be a nondegenerate constant scalar curvature metric. Then, there exists an open neighborhood $U$ of $\overline{g}_{*}$ in $\mathcal{M}^{k, \alpha}(M)_{1}$ such that the set
	$$S = \left\{ \overline{g} \in U : R_{\overline{g}} \mbox{ is constant } \right\},$$
	is a smooth embedded submanifold of $\mathcal{M}^{k, \alpha}(M)_{1}$ which is strongly transverse to the $\mathcal{C}^{k, \alpha}$-normalized conformal classes.
\end{proposition}
\begin{proof}
The proof is a direct application of \cite[Proposition~1]{Uniqueness}.
\end{proof}

\begin{corol}
	Let $\overline{g}_{*} \in \mathcal{M}^{k, \alpha}(M)_{1}$ be a nondegenerate metric on $M$  with constant scalar curvature and vanishing mean curvature. Then, there is an open neighborhood $U$ of $\overline{g}_{*}$ in $\mathcal{M}^{k, \alpha}(M)_{1}$ such that every $\mathcal{C}^{k, \alpha}$-normalized conformal class of metrics in $\mathcal{M}^{k, \alpha}(M)_{1}$ has at most one metric of constant scalar curvature and volume one in $U$.
\end{corol}
\begin{proof}
The fact that the manifold $S$ is transverse to the normalized conformal class guarantees the local uniqueness of intersections.
\end{proof}

\subsubsection{Bifurcation of solutions}

We begin this subsection defining bifurcation instant. To do so, let $M$ be a $m$-dimensional compact Riemannian manifold with boundary, $m \geq 3$. Define
$$
\begin{array}{rcl}
\left[ a, b \right] & \longrightarrow & \mathcal{M}^{k, \alpha}(M)_{1}, \ \ \ k \geq 3\\
s & \longmapsto & \overline{g}_{s}
\end{array}
$$
a continuous path of Riemannian metrics on $M$ having constant scalar curvature $R_{\overline{g}_{s}}$ and vanishing mean curvature $H_{\overline{g}_{s}}$, for all $s \in [a,b]$.

\begin{definition}\label{def:bif}
	An instant $s_{*} \in [a, b]$ is called a \textbf{bifurcation instant} for the family $\{\overline{g}_{s}\}_{s \in [a,b]}$ if there exists a sequence $(s_{n})_{n \geq 1} \subset [a,b]$ and a sequence $(\overline{g}_{n})_{n \geq 1} \subset \mathcal{M}^{k, \alpha}(M)_{1}$ of Riemannian metrics on $M$ satisfying:
	\begin{enumerate}
		\item[(a)] $\displaystyle\lim_{n \rightarrow \infty} s_{n} = s_{*}$ and $\displaystyle\lim_{n \rightarrow \infty} \overline{g}_{n} = \overline{g}_{s_{*}} \in \mathcal{M}^{k, \alpha}(M)_{1}$;
		\item[(b)] $\overline{g}_{n} \in [\overline{g}_{s_{n}}]$, but $\overline{g}_{n} \neq \overline{g}_{s_{n}}$, for all $n \geq 1$;
		\item[(c)] $\overline{g}_{n}$ has constant scalar curvature and vanishing mean curvature, for all $n \geq 1$.
	\end{enumerate}
	If $s_{*} \in [a,b]$ is not a bifurcation instant, the family $\{\overline{g}_{s}\}_{s \in [a,b]}$ is said \textbf{locally rigid} at $s_{*}$.
\end{definition}

An instant $s \in [a,b]$ for which $\frac{R_{\overline{g}_{s}}}{(m-1)}$ is a non-vanishing solution of problem (\ref{EP}) is called a \textbf{degeneracy instant} for the family $\{g_{s}\}_{s \in [a,b]}$. Degeneracy instants are candidates to be bifurcation instants.

Consider $M$, a compact Riemannian $m$-manifold with boundary $\partial M \neq 0$, $m \geq 3$. Let 
$$
\begin{array}{rcl}
\left[ a, b \right] & \longrightarrow & \mathcal{M}^{k, \alpha}(M)_{1}, \ \ \ k \geq 3\\
s & \longmapsto & \overline{g}_{s}
\end{array}
$$
be a continuous path of Riemannian metrics in $M$ (with volume $1$) with constant scalar curvature $R_{\overline{g}_{s}}$ and vanishing mean curvature of the boundary $H_{\overline{g}_{s}}$, for all $s \in [a,b]$. Let $G$ be a connected Lie group, finite dimensional, acting on $M$ by diffeomorphisms that preserve orientation and the metric $\overline{g}_s$, for all $s \in [a,b]$, that is,
$$
\begin{array}{rcl}
\xi : G \times M & \longrightarrow & M \\
(g, x) & \longmapsto & \xi (g,x) = \xi_g(x) = g \cdot x,
\end{array}
$$
is a group action with $G$ a Lie group such that
$$G \subset \displaystyle\bigcap_{s \in [a,b]} Iso (M,\overline{g}_s).$$
Therefore,

\vspace{0.5cm}

\begin{minipage}[r]{5cm}
	\begin{flushright}
		$
		\begin{array}{rclcl}
		\xi_g & : & M & \longrightarrow & M \\
		& & x & \longmapsto & \xi_g(x) = g \cdot x
		\end{array}
		$
	\end{flushright}
\end{minipage}
\hspace{0.5cm}
\begin{minipage}[l]{10cm}
	\begin{flushleft}
		is an isometry in $M$, for all $g \in G$, and
	\end{flushleft}	
\end{minipage}

\vspace{0.5cm}

\begin{minipage}[r]{5cm}
	\begin{flushright}
		$
		\begin{array}{rclcl}
		\xi^x & : & G & \longrightarrow & M \\
		& & g & \longmapsto & \xi^x(g) = g \cdot x
		\end{array}
		$
	\end{flushright}
\end{minipage}
\hspace{0.5cm}
\begin{minipage}[l]{10cm}
	\begin{flushleft}
		is differentiable, for all $x \in M$.
	\end{flushleft}	
\end{minipage}

\vspace{0.5cm}

Let
$$\Delta_{\overline{g}_s} : T_{\overline{g}_s}[\overline{g}_s]_1^0 \longrightarrow \mathcal{C}^{k-2, \alpha}(M),$$
be the Laplacian operator on the manifold $(M, \overline{g}_s)$ restricted to
$$T_{\overline{g}_s}[\overline{g}_s]_1^0 = \left\{ \psi \in \mathcal{C}^{k, \alpha}(M) : \partial_{\eta_s} \psi = 0 \mbox{ e } \int_M \psi \omega_{\overline{g}_s} = 0 \right\}.$$
Observe that, if $f \in T_{\overline{g}_s}[\overline{g}_s]_1^0$, then, for all isometry $\xi \in G$, we have \begin{align}\label{eq:int}
	\int_M (f \circ \xi) \ \omega_{\overline{g}_s} = \int_M (f \circ \xi) \ \xi^*(\omega_{\overline{g}_s}) = \int_M  \xi^* (f \ \omega_{\overline{g}_s}) = \int_M f \ \omega_{\overline{g}_s} = 0,
\end{align}
besides that, as we already show at the Proposition \ref{prop:autesp}, $\partial_{\eta_s} (f \circ \xi) = 0$, so $f \circ \xi \in T_{\overline{g}_s}[\overline{g}_s]_1^0$. 

Remember that $\Delta_{\overline{g}_s} \coloneqq - div_{\overline{g}_s} \nabla$ has infinitely many non-negative eigenvalues and the dimensions of the associated eigenspaces are all finite. Now, for each eigenvalue $\rho$ of $\Delta_{\overline{g}_s}$, solution for the Neumann problem
\begin{equation}
	\left\{
	\begin{array}{rcl}
		\Delta_{\overline{g}_s} f & = & \rho f, \mbox{ on $M$} \\
		\partial_{\eta_s} f & = & 0, \mbox{ on $\partial M$},
	\end{array}
	\right.
	\label{Neumann}
\end{equation}
where $\eta_s$ is the unitary (inward) vector field normal to the boundary of $M$, denote by $V_{s, \rho} \subset  T_{\overline{g}_s}[\overline{g}_s]_1^0 $ the correspondent eigenspace. 

\begin{proposition}
	For each $\rho$, solution of \eqref{Neumann}, the linear aplication
	$$\pi_{s,\rho} : G \longrightarrow GL(V_{s,\rho}),$$
	defined by 
	$$
	\begin{array}{rclcl}
	\pi_{s, \rho}(\xi) & : & V_{s, \rho} & \longrightarrow & V_{s, \rho} \\
	& & f & \mapsto & f \circ \xi
	\end{array}
	$$
	is a (anti-)representation of $G$ in $V_{s, \rho}$. 
\end{proposition}
\begin{proof}
	The result follows from \eqref{eq:int} and from the results of the Section \ref{sec:iso_rep}, where we proved that the eigenspaces of the Laplacian are invariant by isometries (Proposition \ref{prop:autesp}) and that the isometric action of $G$ in $M$ determines (anti-)representations of $G$ in all eigenspaces of the Laplacian (Proposition \ref{prop:repres}). 
\end{proof}

We see that a metric $\overline{g}_s$, which is a critical point of the Hilbert-Einstein functionl restricted to $[\overline{g}]^0_1$ is non-degenerate if and only if $R_{\overline{g}_s}= 0$ or $\frac{R_{\overline{g}_s}}{m-1}$ is not an eigenvalue of the operator $\Delta_{\overline{g}_s}$, with Neumann conditions on the boundary. Proposition~\ref{thm:rigidity} ensures local uniqueness of solution for the Yamabe problem if $\overline{g}_s$ is non-degenerated. So that, the bifurcation instants may occur among the degenerated instants. Therefore, if we want to establish conditions for which there is bifurcation, it is in our interest to study the spectrum of the Laplacian, or equivalently, the spectrum of the operator 
$$\mathcal{J}_s = \Delta_{\overline{g}_s} - \frac{R_{\overline{g}_s}}{m-1}I,$$
where $I$ denotes the identity operator and whose domain is $\left\{ \psi \in \mathcal{C}^{k, \alpha}(M)^0 : \displaystyle\int_{M} \psi \ \omega_{\overline{g}_s} = 0 \right\}$. As the Laplacian eigenvalues are all positive and $0 = \rho_0 < \rho_1 < \ldots < \rho_t < \ldots $, there is a finite number of eigenvalues $\rho \leq \frac{R_{\overline{g}_s}}{m-1}$, that is, there are finitely many negative eigenvalues of the operator $\mathcal{J}_s$. 

Thus, for each $s \in [a,b]$, we can define (anti-)representations of $G$ in the negative eigenspace of $\mathcal{J}_s$, considering the direct sum representation,
$$
\begin{array}{ccccl}
\pi_s^- & : & G & \longrightarrow & GL(V_s^-) \\
& & \xi & \longmapsto & \pi_s^-(\xi)
\end{array},
$$
where
$$\pi^-_s = \displaystyle\bigoplus_{\rho \leq \frac{R_{\overline{g}_s}}{m-1}} \pi_{s, \rho} \ \ \ \ \ \mbox{e} \ \ \ \ \ V^-_s = \displaystyle\bigoplus_{\rho \leq \frac{R_{\overline{g}_s}}{m-1}} V_{s, \rho},$$
and the linear transformation
$$
\begin{array}{ccccl}
\pi_s^-(\xi) & : & V_s^- & \longrightarrow &V_s^- \\
& & f & \longmapsto & \pi_s^-(\xi) f
\end{array},
$$
given by
$$\pi_s^-(\xi) f  = \pi_{s, \rho_{0}}(\xi) f_0 \oplus \ldots \oplus \pi_{s, \rho_{r}}(\xi) f_r,$$
with $f = f_0 \oplus \ldots \oplus f_r \in V_s^-$,  where $f_i \in V_{s, \rho_i}$, for all $i=0, ..., r$.

The following result establishes the bifurcation condition for the Yamabe problem in manifolds with boundary. This equivariant bifurcation theorem for manifolds with minimal boundary is a natural extension of the \cite[Teorema~3.4]{Lima} and the proof is essentially the same. Here we transcribe a detailed proof, valid for the case of manifolds with boundary. The result follows from the equivariant bifurcation theorem for smoothly varying domains, proved in \cite[Appendix~A.2]{Lima}, that is a slightly different version of Smoler and Wasserman \cite{Wasserman} theorem, so the proof consists in verify its eight assumptions. First, let us give the following definition.

\begin{definition}
	Given a Banach space $B$, a family of Banach submanifolds of $B$, $[a,b] \ni s \mapsto B_s$ is called a $\mathcal{C}^1$-family of submanifolds of $B$, if the set
	$$\mathcal{B} = \left\{ (x, s) \in B \times [a,b] : x \in B_s \right\}$$
	has the structure of a subbundle $\mathcal{C}^1$ of the trivial bundle $B \times [a,b]$. Analogously, we can define a $\mathcal{C}^1$-family of closed subspaces of $B$ as a family $[a,b] \ni s \mapsto S_s$ of subspaces of the Banach space $B$ such that the set
	$$\mathcal{S} = \left\{ (x, s) \in B \times [a,b] : x \in S_s \right\}$$
	is a subbundle of the Banach trivial bundle $B \times [a,b]$.
\end{definition}
Note that, with the notations established in the above definition, if $s \mapsto x_s \in B$ is a $\mathcal{C}^1$-path, $\mathcal{B} = \bigcup_{s \in [a,b]} (B_s \times \{s\})$ is a $\mathcal{C}^1$-family of submanifolds of $B$, with $x_s \in B_s$ for all $s$, and the path $s \mapsto T_{x_s} S_s$ is a $\mathcal{C}^1$-family of closed subspaces of $B$.

\begin{theorem}\label{thm:bifeq}
	In the situation described above, suppose the family $\{\overline{g}_s\}_{s \in [a,b]}$ is locally rigid in the extremes of the interval $[a,b]$ and that the representations $\pi^-_a$ e $\pi^-_b$ are non-equivalent. Then, there exists a bifurcation instant $s_*$ for the family $\{\overline{g}_s\}_{s \in [a,b]}$ in the interval $(a,b)$.  
\end{theorem}
\begin{proof}
	Let $B_0 = \mathcal{C}^{k-2, \alpha}(M)$, $B_2 = \mathcal{C}_+^{k, \alpha}(M)$ be Banach spaces and $\mathcal{H} = L^2(M)$ the space of square integrable functions with respect to the measure induced by the volume form on $M$ associated to any of the metrics $\overline{g}_s$.  Consider an action of $G$ on these spaces via pull-back and define the vector bundles
	$$\mathcal{D}^0 = \left\{ (\phi, s) \in \mathcal{C}^{k, \alpha}_+(M) \times [a,b] : \displaystyle\int_M \phi^{\frac{m}{2}} \ \omega_{\overline{g}_s} = 1, \ \partial_{\eta_s} \phi = 0  \right\},$$
	and
	$$\mathcal{E}^0 = \left\{ (\psi, s) \in \mathcal{C}^{k-2, \alpha}(M) \times [a,b] : \displaystyle\int_M \psi \ \omega_{\overline{g}_s} = 0, \ \partial_{\eta_s} \psi = 0 \right\},$$
	on the interval $[a,b]$, where $\eta_s$ is the unitary (inward) vector field normal to the boundary of $M$, with respect to the metric $\overline{g}_s$.
	
	The set
	$$\mathcal{E}_s^0 = \left\{ \psi \in \mathcal{C}^{k-2, \alpha}(M) : \displaystyle\int_M \psi \ \omega_{\overline{g}_s} = 0, \ \partial_{\eta_s} \psi = 0 \right\}$$
	is closed in $\mathcal{C}^{k-2, \alpha}(M)$ and it contains the null function, so it is a Banach subspace of $B_0$. In particular, $\mathcal{E}^0$ is a subbundle of the trivial bundle $\mathcal{C}^{k-2, \alpha}(M) \times [a,b]$, hence $s \mapsto \mathcal{E}^0_s$ is a $\mathcal{C}^1$-family of closed subspaces of the Banach space $B_0$. Observe that $G$ acts on $\mathcal{M}^{k,\alpha}(M)_1$, via pull-back: for each isometry $\xi$ and each metric $\tilde{g} \in \mathcal{M}^{k,\alpha}(M)_1$, we have $\xi^*{\tilde{g}} =\tilde{g}$. Therefore, for all $\psi \in \mathcal{E}^0_s$, we have 
	$$ \int_M \psi \circ \xi \ \omega_{\overline{g}_s} = \int_M \xi^*(\psi \ \omega_{ \overline{g}_s}) =  \int_M \psi \ \omega_{\overline{g}_s} = 0,$$
	and $$\partial_{\eta_s} (\xi \circ \psi) = \partial_{\eta_s} \xi (\psi) \cdot \partial_{\eta_s} \psi = 0,$$
	so,  $\mathcal{E}^0_s$ é $G$-invariant.
	
	Analogously, we define
	$$\mathcal{D}_s^0 = \left\{ \phi \in \mathcal{C}_{+}^{k, \alpha}(M) : \displaystyle\int_M \phi^{\frac{m}{2}} \ \omega_{\overline{g}_s} = 1, \ \partial_{\eta_s} \phi = 0 \right\}.$$		
	Observe that, for each $s$,  $\mathcal{D}_s^0$ can be identified with $[\overline{g}_s]^0_1$ that is an embedded submanifold of $[\overline{g}_s]$, that can be identified with $B_2$. So, $\mathcal{D}_s^0$ is a submanifold of $B_2$. As $\mathcal{D}^0$ is a subbundle of the trivial bundle $B_2 \times [a,b]$, it follows that $\mathcal{D}^0_s$ is a $\mathcal{C}^1$-family of submanifolds of $B_2$. The action of $G$ leaves invariant $\mathcal{D}^0$, because for all $\phi \in \mathcal{D}^0_s$, $\xi \in G$, we have $\partial_{\eta_s} (\xi \circ \psi) = 0,$ and
	$$ \int_M (\phi \circ \xi)^{\frac{m}{2}} \ \omega_{\overline{g}_s} =  \int_M (\phi^{\frac{m}{2}} \circ \xi) \ \omega_{\overline{g}_s} = \int_M \xi^*(\phi^{\frac{m}{2}} \ \omega_{\overline{g}_s}) =  \int_M \phi^{\frac{m}{2}} \ \omega_{\overline{g}_s} = 1.$$
	
	Define closed subespace of $\mathcal{H}$,
	$$\mathcal{H}_s = \left\{ \kappa \in L^2(M) : \displaystyle\int_M \kappa \ \omega_{\overline{g}_s} = 0 \right\},$$
	and note that, for each $s \in [a,b]$ there exists a complete inner product, given by
	\begin{equation}\label{innerL2}
		\langle \kappa_1, \kappa_2 \rangle_s = \displaystyle\int_M \kappa_1 \cdot \kappa_2 \ \omega_{\overline{g}_s},
	\end{equation}
	for each $ \kappa_1, \kappa_2 \in \mathcal{H}_s$. Similarly to what we did for $\mathcal{E}^0_s$, it can be shown that $\mathcal{H}_s$ is $G$-invariant.
	
	Now, we define the application
	$$
	\begin{array}{rrccl}
	F & : & \mathcal{D}^0 & \longrightarrow & \mathcal{E}^0 \\
	& & (\phi, s) & \longmapsto & \left( R_{\phi \overline{g}_s} - \displaystyle\int_M R_{\phi \overline{g}_s} \ \omega_{\overline{g}_s}, s \right);
	\end{array}
	$$
	that is a fiber bundle morphism. 
	
	Note that $s \mapsto \phi_s \in \mathcal{D}^0_s$ is a $\mathcal{C}^1$-section of the fiber bundle $\mathcal{D}^0$ and $s \mapsto \psi_s \in \mathcal{E}^0_s$ is a $\mathcal{C}^1$-section of the bundle $\mathcal{E}^0$ and $F(\phi, s) = (\psi, s)$, if $(\phi, s) \in \mathcal{D}^0$. Now, the metric $\phi \overline{g}_s \in [\overline{g}_s]$ has constant scalar curvature if and only if $F(\phi, s) = (\mathbf{0}, s ) = \mathbf{0}_s$, where, for each $s$, $\mathbf{0}_s$ is the null function in $\mathcal{E}^0_s$. Hence, we are interested in the inverse image by $F$ of the null section of the fiber bundle $\mathcal{E}^0$, which we denote by
	$$F^{-1}(\mathbf{0}_s) = \left\{(\phi, s) \in \mathcal{D}^0 : F(\phi, s) = (\mathbf{0}, s) \right\}.$$
	Observe that the constant section 
	$$
	\begin{array}{ccccl}
	\mathbbm{1}_{\mathcal{D}^0} & : & [a,b] & \longrightarrow & \mathcal{D}^0_s \\ [0,2cm]
	& & s & \longmapsto & \mathbbm{1}_{\mathcal{D}^0}(s) = \mathbf{1}_s,
	\end{array}
	$$
	that associates to each $s \in [a,b]$ the constant function equals to $1$, denoted by $\mathbf{1}_s : M \rightarrow \mathbb{R}$, in $\mathcal{D}^0_s$, belongs to $ F^{-1}(\mathbf{0}_s)$, since by hypothesis $\overline{g}_s$ has constant scalar curvature, for each $s \in [a,b].$
	The orbits of $\mathbf{1}_s$ and $\mathbf{0}_s$ by the action of $G$ on $\mathcal{E}^0_s$ and $\mathcal{D}^0_s$ (induced by the action of $G$ on $B_0$ and on $B_2$) are trivial -- $\mathbf{1}_s \cdot \xi = \mathbf{1}_s$ -- and $\mathbf{0}_s \circ \xi = \mathbf{0}_s$, that is, $G$ fixes the orbits of $\mathbf{1}_s$ and of $\mathbf{0}_s$.
	The tangent space of $\mathcal{D}^0_s$ in $\mathbf{1}_s$ is given by
	$$T_{\mathbf{1}_s}\mathcal{D}_s^0 = \left\{ \psi \in \mathcal{C}^{k, \alpha}(M) : \displaystyle\int_M \psi \ \omega_{\overline{g}_s} = 0, \ \partial_{\eta_s} \psi = 0 \right\}.$$
	Remember the set
	$$\mathcal{C}^{k, \alpha}(M)^0  = \left\{ \psi \in \mathcal{C}^{k, \alpha}(M) : \partial_{\eta_s} \psi = 0  \right\};$$
	then, the inclusions
	$$\mathcal{C}^{k, \alpha}(M)^0  \subset \mathcal{C}^{k, \alpha}(M) \subset \mathcal{C}^{k-2, \alpha}(M) \subset L^2(M)$$
	induce inclusions
	$$T_{\mathbf{1}_s}\mathcal{D}_s^0 \subset \mathcal{E}_s^0 \subset \mathcal{H}_s,$$
	for all $s$ in the interval $[a,b]$.
	
	Let us show that $F_s = F(\cdot , s) : \mathcal{D}^0_s \rightarrow \mathcal{E}^0_s$ is a gradient operator for all $s \in [a,b]$, that is, the differential $(dF_s)_{\mathbf{1}_s} : T_{\mathbf{1}_s}\mathcal{D}^0_s \rightarrow \mathcal{E}^0_s$ is symmetric with respect to a complete inner product in $\mathcal{H}_s$ for each $s \in [a,b]$. Indeed, for each $\psi \in T_{\mathbf{1}_s}\mathcal{D}^0_s$,  $(dF_s)_{\mathbf{1}_s}$ is given by
	\begin{align*}
		\left( dF_s \right)_{\mathbf{1}_s}(\psi)
		& = \left. \displaystyle\frac{d}{dt} \right|_{t=0} F_s (\mathbf{1}_s + t \psi) \\
		& = \left. \displaystyle\frac{d}{dt} \right|_{t=0}  \left( R_{(\mathbf{1}_s + t \psi) \overline{g}_s } - \displaystyle\int_M R_{(\mathbf{1}_s + t \psi) \overline{g}_s } \omega_{\overline{g}_s} \right) \\
		& =\left. \displaystyle\frac{d}{dt} \right|_{t=0} \left( R_{(\overline{g}_s + t \psi \overline{g}_s) } - \displaystyle\int_M R_{(\overline{g}_s + t \psi \overline{g}_s) } \omega_{\overline{g}_s} \right)
	\end{align*}
	Observe that $\overline{g}_s + t \psi \overline{g}_s = \overline{g}_s(t)$ is a variation of $\overline{g}_s$ in the direction of the tensor $h_s  = \psi \overline{g}_s \in T_{\overline{g}}[\overline{g}]^0_1$. Then, 
	we have
	\begin{align*}
		\left. \displaystyle\frac{d}{dt} \right|_{t=0} R_{\overline{g}_s(t)} 
		& = - \psi \overline{g}_s^{ij} R_{ij} + \nabla_i (\nabla_j \psi \overline{g}_s^{ij} - \nabla^i (m \psi)) \\ 
		& = - \psi \overline{g}_s^{ij} R_{ij} + \nabla_i \left(\nabla^i \psi - m \nabla^i\psi \right) \\ 
		& =  - \psi R_{\overline{g}_s} - \Delta_{\overline{g}_s} \psi + m \Delta_{\overline{g}_s} \psi \\ 
		& =  (m-1) \Delta_{\overline{g}_s} \psi - \psi R_{\overline{g}_s},
	\end{align*}
	where we are considering the geometric Laplacian, $\Delta \psi = - \nabla_i \nabla^i \psi$, and using Einstein's notation; so
	\begin{align}\label{int_null}
		\begin{split}
			\left. \displaystyle\frac{d}{dt} \right|_{t=0} \displaystyle\int_M R_{\overline{g}_s(t)} \ \omega_{\overline{g}_s} 
			& = \displaystyle\int_M \left( (m-1) \Delta_{\overline{g}_s} \psi - \psi R_{\overline{g}_s} \right) \ \omega_{\overline{g}_s} \\
			& = (m-1) \displaystyle\int_M \Delta_{\overline{g}_s} \psi \ \omega_{\overline{g}_s} - \underbrace{R_{\overline{g}_s}}_{cte}  \underbrace{\displaystyle\int_M  \psi \ \omega_{\overline{g}_s}}_{= 0} \\
			& = (m-1) \displaystyle\int_{\partial M} \langle \nabla \psi , \eta_s \rangle \ \omega_{\overline{g}_s} \\
			&= (m-1) \displaystyle\int_{\partial M} \partial_{\eta_s} \psi \ \omega_{\overline{g}_s} = 0
		\end{split}
	\end{align}
	
	Therefore, $(d F_s)_{\mathbf{1}_s} (\psi) = (m-1) \Delta_{\overline{g}_s} \psi - R_{\overline{g}_s} \psi$, that can be expressed as 
	$$(d F_s)_{\mathbf{1}_s} (\psi) = (m-1) \left( \Delta_{\overline{g}_s} \psi - \displaystyle\frac{R_{\overline{g}_s}}{m-1} \psi \right).$$
	
	Now, the linear operator $\mathcal{J}_s = \Delta_{\overline{g}_s} - \displaystyle\frac{R_{\overline{g}_s}}{m-1} I$, defined in $T_{\mathbf{1}_s} \mathcal{D}^0_s$, is symmetric with respect to the inner product $\langle \ \cdot \ , \ \cdot \ \rangle_s$. Indeed, if $\psi, \varphi \in T_{\mathbf{1}_s} \mathcal{D}^0_s$, then
	\begin{align*}
		\left\langle \mathcal{J}_s \psi , \varphi \right\rangle_s 
		& = \displaystyle\int_M \left( \Delta_{\overline{g}_s} \psi - \displaystyle\frac{R_{\overline{g}_s}}{m-1} \psi \right) \varphi \ \omega_{\overline{g}_s} \\
		& = \displaystyle\int_M \left( \Delta_{\overline{g}_s} \psi  \right) \varphi \ \omega_{\overline{g}_s} - \displaystyle\int_M {\left( \displaystyle\frac{R_{\overline{g}_s}}{m-1} \psi \right) \varphi} \ \omega_{\overline{g}_s} \\
		& = \displaystyle\int_M {\psi \left( \Delta_{\overline{g}_s} \varphi  \right)} \ \omega_{\overline{g}_s} - \displaystyle\int_M {\psi \left( \displaystyle\frac{R_{\overline{g}_s}}{m-1} \varphi \right) } \ \omega_{\overline{g}_s} \\
		& =  \displaystyle\int_M \psi \left( \Delta_{\overline{g}_s} \varphi - \displaystyle\frac{R_{\overline{g}_s}}{m-1} \varphi \right) \ \omega_{\overline{g}_s} \\
		& =  \left\langle \psi , \mathcal{J}_s \varphi \right\rangle_s.
	\end{align*}
	Hence, $F_s$ is a gradient operator, as we want.
	
	Besides that, the application $F$ is $G$-equivariant, that is, 
	$$F(\phi \cdot \xi, s) = F(\phi, s) \cdot \xi,$$
	or, equivalently, $F_s(\phi \cdot \xi) = F_s(\phi) \cdot \xi,$ for all $(\phi, s) \in \mathcal{D}^0$.
	Truly, we have
	\begin{align*}
		F_s(\phi \cdot \xi)
		& = F_s(\phi \circ \xi) \\
		& = R_{(\phi \circ \xi) \overline{g}_s} - \displaystyle\int_M R_{(\phi \circ \xi) \overline{g}_s} \ \omega_{\overline{g}_s} ;
	\end{align*}
	on the other hand,  $F_s(\phi) \cdot \xi = F_s(\phi) \circ \xi = \xi^* F_s(\phi)$; it follows from $R_{(\phi \circ \xi)\overline{g}_s} = R_{\phi \overline{g}_s}\circ\xi $ that
	$$F_s(\phi \cdot \xi) = R_{(\phi \circ \xi) \overline{g}_s} - \displaystyle\int_M R_{(\phi \circ \xi) \overline{g}_s} \ \omega_{\overline{g}_s} = \underbrace{R_{\phi \overline{g}_s}\circ\xi}_{\xi^*R_{\phi \overline{g}_s}} - \displaystyle\int_M \underbrace{R_{\phi \overline{g}_s}\circ\xi \ \omega_{\overline{g}_s}}_{\xi^*(R_{\phi \overline{g}_s} \omega_{\overline{g}_s})} = F_s(\phi) \cdot \xi,$$
	and $F$ is $G$-equivariant.
	
	Observe that
	$$(d F_s)_{\mathbf{1}_s} (\psi) = (m-1) \Delta_{\overline{g}_s} \psi - R_{\overline{g}_s} \psi,$$
	maps $T_{\mathbf{1}_s}\mathcal{D}^0_s$ in $\mathcal{E}^0_s$. Indeed, $\partial_{\eta_s} \left( (m-1) \Delta_{\overline{g}_s} \psi - \psi R_{\overline{g}_s} \right) = 0$, for all $\psi \in T_{\mathbf{1}_s}\mathcal{D}_s^0$, and,
	\begin{align*}
		\displaystyle\int_M  (m-1) \Delta_{\overline{g}_s} \psi - R_{\overline{g}_s} \psi \ \omega_{\overline{g}_s} 
		& =  (m-1) \displaystyle\int_M \Delta_{\overline{g}_s} \psi \ \omega_{\overline{g}_s} - \underbrace{R_{\overline{g}_s}}_{cte}  \underbrace{\displaystyle\int_M  \psi \ \omega_{\overline{g}_s}}_{= 0} \\
		& = (m-1) \displaystyle\int_{\partial M} \underbrace{\langle \nabla \psi , \eta_s \rangle}_{\partial_{\eta_s} \psi = 0} \omega_{\overline{g}_s} = 0.
	\end{align*} Now, 
	$$(d F_s)_{\mathbf{1}_s}: T_{\mathbf{1}_s}\mathcal{D}^0_s \longrightarrow \mathcal{E}^0_s$$
	is Fredholm of index zero. In fact, it is a classic fact from analysis, that the Laplacian operator $\Delta_{\overline{g}}$ is Fredholm of index zero as an operator from $\mathcal{C}^{k, \alpha}(M)$ into the space $\mathcal{C}^{k-2, \alpha}(M)$. Besides that, the operator $\frac{R_{\overline{g}_s}}{m-1} I$ is compact from $\mathcal{C}^{k, \alpha}(M)$ into $\mathcal{C}^{k - 2, \alpha}(M)$ (because $k + \alpha > k -2 + \alpha$ implies that the inclusion $i : \mathcal{C}^{k, \alpha}(M) \lhook\joinrel\longrightarrow \mathcal{C}^{k - 2, \alpha}(M) $ is compact). Now, as $T_{\mathbf{1}_s}\mathcal{D}_s$ is a closed subspace of  $\mathcal{C}^{k, \alpha}(M)$, $\mathcal{E}_s$ is a closed subspace of $\mathcal{C}^{k-2, \alpha}(M)$, and
	$$T_{\bf{1}}\mathcal{D}_s \overset{i}{\lhook\joinrel\longrightarrow} \mathcal{C}^{k, \alpha}(M) \overset{\Delta_{\overline{g}_s}+\frac{R_{\overline{g}_s}}{{m-1}}}{\longrightarrow} \mathcal{C}^{k-2, \alpha}(M) \overset{p}{\longrightarrow} \mathcal{E}_s,$$
	where
	\begin{itemize}
		\item $\mbox{dim ker}(i) = 0, \ \mbox{codim im}(i) = 1 \Longrightarrow \mbox{ind }(i) = -1,$
		\item $\mbox{dim ker}(p) = 1, \ \mbox{codim im}(p) = 0 \Longrightarrow \mbox{ind }(p) = 1,$
	\end{itemize}
	it follows that $\Delta_{\overline{g}_s} +\displaystyle\frac{R_{\overline{g}_s}}{{m-1}} : T_{\mathbf{1}_s}\mathcal{D}_s \longrightarrow \mathcal{E}_s$ is Fredholm of index zero, as we want.
	
	We know that there exists an orthonormal basis $\beta_s$ of eigenfunctions of the Laplacian $\Delta_{\overline{g}_s}$ for the Hilbert space $L^2(M, \overline{g}_s)$ with respect to the inner product \eqref{innerL2}, besides, $\mathcal{H}_s$ is a closed subspace of $L^2(M, \overline{g}_s)$ which is orthogonal to the $1$-dimensional subspace of the constant functions:
	$$L^2(M,\overline{g}_s) = \mathcal{H}_s \oplus \{ \iota \in L^2(M, \overline{g}_s) : \iota \mbox{ é constante } \}.$$
	The subspace of the constant functions is generated by the eigenfunction equals to $1$, associated to the null eigenvalue of the Laplacian, so $\beta_s \backslash \{\mathbf{1}\}$ is a basis for $\mathcal{H}_s$ formed by eigenfunctions of the Laplacian. Observe that, if $\lambda$ is an eigenvalue of $\Delta_{\overline{g}_s}$, with the Neumann boundary condition, then $\lambda - \frac{R_{\overline{g}_s}}{m-1}$ is an eigenvalue of $\Delta_{\overline{g}_s} - \frac{R_{\overline{g}_s}}{m-1}$ and $f$ is eigenfunction of $\Delta_{\overline{g}_s}$ associated to the eigenvalue $\lambda$ if and only if $f$ is an eigenfunction of $\Delta_{\overline{g}_s} - \frac{R_{\overline{g}_s}}{m-1}$ associated to the eigenvalue $\lambda - \frac{R_{\overline{g}_s}}{m-1}$. In particular, $\mathbf{1}$ is an eigenfunction associated to the eigenvalue $-\frac{R_{\overline{g}_s}}{m-1}$ of the operator $\mathcal{J}_s$. Hence, $\beta_s \backslash \mathbf{1}$ is a basis for $\mathcal{H}_s$ formed by eigenfunctions of $\mathcal{J}_s$. 
	
	Once the eigenvalues of $\Delta_{\overline{g}_s}$ are all real, non-negative, with finite multiplicity and the spectrum of $\mathcal{J}_s$ is given by 
	$$\Sigma(\mathcal{J}_s) = \left\{ \lambda - \displaystyle\frac{R_{\overline{g}_s}}{m-1} : \lambda \in \Sigma(\Delta_{\overline{g}_s}) \right\},$$
	it follows that $\mathcal{J}_s$ has finitely many negative eigenvalues, which implies that $(\mbox{d}F_s)_{\mathbf{1}_s}$ has finitely many negative eigenvalues.
	
	We already proved that the eigenspaces $V_{\rho, s}$ of the Laplacian $\Delta_{\overline{g}_s}$ are $G$-invariants. It is easy to see that, $V_{\rho, s}$ as eigenspace of $\mathcal{J}_s : T_{\mathbf{1}} \mathcal{D}_s^0 \longrightarrow \mathcal{E}_s^0$ or, equivalently, $(\mbox{d}F_s)_{\mathbf{1}_s}$ are both $G$-invariants.
	
	Note that $(\mbox{d}F_a)_{\mathbf{1}_a}$ is an isomorphism from $T_{\mathbf{1}_s}\mathcal{D}^0_s$ into $\mathcal{E}^0_s$ because, by hypothesis, $R_{g_a} = 0$ or $\frac{R_{g_a}}{m-1}$ is not an eigenvalue of $\Delta_{g_a}$, which implies that $(\mbox{d}F_a)_{\mathbf{1}_a}$ is injective. Now, we saw that $(\mbox{d}F_a)_{\mathbf{1}_a}$ is Fredholm of incex zero, hence injection implies surjection.  The same is valid for $(\mbox{d}F_b)_{\mathbf{1}_b}$.
	
	By hypothesis, $\pi^-_a$ e $\pi^-_b$ are non-equivalents. Under that conditions, \cite[Theorem~A.3]{Lima} ensures the existence of a bifurcation instant $s_* \in (a,b)$ for the family of solutions, $s \mapsto \phi_s \in \mathcal{D}^0_s$, of the equation 
	\begin{equation}\label{eq:sol}
		F( \ \cdot \ , s) = (\mathbf{0}_s, s).
	\end{equation}
	It means, first, that $s_* \mapsto \phi_{s_*} \in \mathcal{D}^0_{s_*}$ belongs to the family of solutions of \eqref{eq:sol}, besides that, the existence of the bifurcation instant $s_*$, implies the existence of
	\begin{itemize}
		\item[(a.)] a sequence $s_n \in [a,b]$, with $s_n \rightarrow s_*$; i.e., there exists a sequence of functions $\phi_{s_n} \in \mathcal{D}^0_{s_n}$ that are solutions of the equation \eqref{eq:sol}, such that $\phi_{s_n} \rightarrow \phi_{s_*}$;
		\item[(b.)] a sequence of solutions of \eqref{eq:sol} with $\phi_n \in \mathcal{D}^0_{s_n}$, $\phi_n \rightarrow \phi_{s_*}$ and $\phi_n \neq \phi_{s_n}$, for all $n \in \mathbb{N}$.
	\end{itemize}
	This is equivalent to say that $s_*$ is a bifurcation instant for the family of solutions of the Yamabe problem in manifolds with minimal boundary $s \mapsto \overline{g}_s \in [\overline{g}_s]^0_1$, in the sense of Definition~\ref{def:bif}.	
\end{proof}

\section{Hamonically free actions} \label{sec:harm}

Let $(M_1, g^{(1)})$ be a compact Riemannian manifold, with $\partial M_{1} = \emptyset$ and constant scalar curvature, and let $(M_{2}, \overline{g}^{(2)})$ be a compact Riemannian manifold with minimal boundary and constant scalar curvature. Consider the product manifold, $M = M_{1} \times M_{2}$, which boundary is given by \linebreak $\partial M = M_{1} \times \partial M_{2}$. Let $m_{1}$ and $m_{2}$ be the dimensions of $M_{1}$ and $M_{2}$, respectively, and assume that $\mbox{dim}(M) = m = m_{1} + m_{2} \geq 3$. For each $s \in (0, + \infty)$ define ${\overline{g}}_{s} = {g}^{(1)} \oplus s {\overline{g}}^{(2)}$ a family of metrics in $M$.  Then, $\left\{ {\overline{g}}_{s} \right\}_{s}$ is a family of metrics of constant scalar curvature and minimal boundary in $\mathcal{M}^{k, \alpha}(M)$, that is, normalizing $\overline{g}_s$ in such a way that it has volume $1$, for all $s>0$, we have a family of critical points of the Hilbert-Einstein functional restricted to $[\overline{g}]^0_1$, so a family of solutions for the Yamabe problem in manifolds with boundary.

As we saw in the preceding section, a necessary condition for an instant $s$ to be a bifurcation instant for the family $\{\overline{g}_s\}$ is that $\overline{g}_s$ is a degenerate metric. In Subsection \ref{sec:HEfunc} we saw that a metric is degenerate if and only if $\frac{R_{\overline{g}}}{m-1}$ is a non-null eigenvalue of $\Delta_{\overline{g}_s}$, with Neumann boundary condition. It is equivalent to say that zero is an eigenvalue of the operator
$$\mathcal{J}_{s} = \Delta_{{\overline{g}}_{s}} - \displaystyle\frac{R_{{\overline{g}}_{s}}}{m-1},$$
defined in $T_{\overline{g}}[\overline{g}]^0_1 = \left\{ \psi \in \mathcal{C}^{k, \alpha}(M)^0 : \displaystyle\int_{M} \psi \ \omega_{\overline{g}_s} = 0 \right\}$, satisfying the Neumann boundary conditions.

Let $\Delta_{{g}^{(1)}}$ be the Laplace operator on $M_1$, which eigenvalues are denoted by $0 = \rho_{0}^{(1)} < \rho_{1}^{(1)} < \rho_{2}^{(1)} < \ldots $, with geometric multiplicity $\mu_{i}^{(1)}$, $i \geq 0$, and let $\Delta_{{\overline{g}}^{(2)}}$ be the Laplace operator of $M_2$ which eigenvalues, subjected to Neumann boundary condition, 
\begin{equation}
\left\{
\begin{array}{rcl}
\Delta_{{\overline{g}}^{(2)}} f^{(2)} & = & \rho_{j}^{(2)} f^{(2)},\mbox{ on $ M$,} \\
\partial_{\eta_2} f^{(2)} & = & 0, \mbox{ on $\partial M$},
\end{array}
\right.
\label{EP2}
\end{equation}
are $0 = \rho_{0}^{(2)} < \rho_{1}^{(2)} < \rho_{2}^{(2)} < \ldots $,where $j \geq 0$, with $\mu_{j}^{(2)}$ the geometric multiplicity of $\rho_{j}^{(2)}$, $j  \geq 0$. Recall that, in the product manifold $M$, the Laplacian is given by
$$\Delta_{\overline{g}_s} = \Delta_{g^{(1)}} \otimes I + \displaystyle\frac{1}{s} \left( I \otimes \Delta_{\overline{g}^{(2)}}\right)$$
Then, the spectrum of $\mathcal{J}_{s}$ is given by
$$\Sigma(\mathcal{J}_{s}) = \left\{ \sigma_{i,j} : i, j \geq 0, i + j > 0 \right\},$$
where
$$\sigma_{i,j}(s) = \rho_{i}^{(1)} + \displaystyle\frac{1}{s} \rho_{j}^{(2)} - \displaystyle\frac{1}{m-1} \left( R_{{g}^{(1)}} + \displaystyle\frac{1}{s} R_{{\overline{g}}^{(2)}} \right)$$
are eigenvalues of $J_{s}$, with Neumann boundary condition on $\partial M$, with geometric multiplicity equal to the product $\mu_{i}^{(1)} \mu_{j}^{(2)}$. Besides that, all the eigenvalues of $J_s$ are of this form (see Subsection \ref{sec:laplacian}). It is worth verifying that $\sigma_{i,j}(s)$ is a solution of the Neumann problem
\begin{equation} \label{eq:EJ}
\left\{
\begin{array}{rcl}
\mathcal{J}_s f & = & \sigma_{i,j}(s) f,\mbox{ on $ M$,} \\
\partial_{\eta_s} f & = & 0, \mbox{ on $\partial M$}.
\end{array}
\right.
\end{equation}
Let $f^{(1)}_i$ be an eigenfunction of $\Delta_{g^{(1)}}$ associated to the eigenvalue $\rho^{(1)}_i$ and let $f^{(2)}_j$ be an eigenfunction associated to $\rho^{(2)}_j$, solution of the problem \eqref{EP2}, then $f = f^{(1)}_i \otimes f^{(2)}_j$ is an eigenfunction of $\mathcal{J}_s$, associated to the eigenvalue $\sigma_{i,j}(s)$ and satisfies \eqref{eq:EJ}. Indeed,
\begin{align*}
\mathcal{J}_s f 
&= \left( \Delta_{{\overline{g}}_{s}} - \displaystyle\frac{R_{{\overline{g}}_{s}}}{m-1} \right) f \\
&= \Delta_{g^{(1)}} \otimes I + \displaystyle\frac{1}{s} \left( I \otimes \Delta_{\overline{g}}^{(2)} \right) \left(f_i^{(1)} \otimes f_j^{(2)} \right) - \left( \displaystyle\frac{R_{{\overline{g}}_{s}}}{m-1} \right) \left(f_i^{(1)} \otimes f_j^{(2)}\right) \\
&= \Delta_{g^{(1)}} f_i^{(1)} \otimes f_j^{(2)} + \displaystyle\frac{1}{s} \left( f_i^{(1)} \otimes   \Delta_{\overline{g}}^{(2)} f_j^{(2)} \right) - \left( \displaystyle\frac{R_{{\overline{g}}_{s}}}{m-1} \right) \left(f_i^{(1)} \otimes f_j^{(2)}\right) \\
&= \rho_i^{(1)} f_i^{(1)} \otimes f_j^{(2)} + \displaystyle\frac{1}{s}  f_i^{(1)} \otimes \left( \rho_j^{(2)} f_j^{(2)} \right) - \left( \displaystyle\frac{R_{{\overline{g}}_{s}}}{m-1} \right) \left(f_i^{(1)} \otimes f_j^{(2)}\right) \\
&= \rho_i^{(1)} \left( f_i^{(1)} \otimes f_j^{(2)} \right) +  \displaystyle\frac{1}{s}  \rho_j^{(2)} \left( f_i^{(1)} \otimes  f_j^{(2)} \right) -  \displaystyle\frac{1}{m-1} \left( R_{{g}^{(1)}} + \displaystyle\frac{1}{s} R_{{\overline{g}}^{(2)}} \right) \left(f_i^{(1)} \otimes f_j^{(2)}\right) \\
& = \left( \rho_{i}^{(1)} + \displaystyle\frac{1}{s} \rho_{j}^{(2)} \right)  (f_i^{(1)} \otimes f_j^{(2)})  - \displaystyle\frac{1}{m-1} \left( R_{{g}^{(1)}} + \displaystyle\frac{1}{s} R_{{\overline{g}}^{(2)}} \right) (f_i^{(1)} \otimes f_j^{(2)}) \\
&= \left[ \rho_{i}^{(1)} + \displaystyle\frac{1}{s} \rho_{j}^{(2)} - \displaystyle\frac{1}{m-1} \left( R_{{g}^{(1)}} + \displaystyle\frac{1}{s} R_{{\overline{g}}^{(2)}} \right) \right] \left( f^{(1)}_i \otimes f^{(2)}_j \right) \\
&= \sigma_{i,j}(s) f
\end{align*}
and, we have
\begin{align*}
\partial_{\eta_s} f
&= \eta_s \left( f^{(1)}_i \otimes f^{(2)}_j \right) 
= \left( 0 + \frac{1}{\sqrt{s}} \eta_{\overline{g}^{(2)}} \right) \left( f^{(1)}_i \otimes f^{(2)}_j \right) \\
&= \frac{1}{\sqrt{s}} \eta_{\overline{g}^{(2)}} f^{(2)}_j 
= \frac{1}{\sqrt{s}}  \partial_{\eta_s} f^{(2)}_j 
= 0
\end{align*}
Despite the eigenvalues $\rho_i^{(1)}$, $\rho^{(2)}_j$ are all distinct, the eigenvalues $\sigma_{i,j}(s)$ of $\mathcal{J}_s$ are not necessarily distinct.

\begin{definition}\label{def:degen}
	{\rm Let $i_{*}$ and $j_{*}$ be the smallest non-negative integers that satisfy
		$$
		\rho_{i_{*}}^{(1)} \geq \displaystyle\frac{R_{{g}^{(1)}}}{m-1}, \ \ \ \rho_{j_{*}}^{(2)} \geq \displaystyle\frac{R_{{\overline{g}}^{(2)}}}{m-1}.
		$$
		We say that the pair of metrics $({g}^{(1)}, {\overline{g}}^{(2)})$ is {\bf degenerate} if equalities hold in both cases, that is, $\sigma_{i_{*}, j_{*}}(s) = 0$, for all $s$. In this situation, the operator $\mathcal{J}_{s}$ is also called degenerate.}
\end{definition}

In \cite{Diaz}, we prove that, if the pair $({g}^{(1)}, {\overline{g}}^{(2)})$ is non-degenerate, and $R_{{g}^{(1)}}, R_{{\overline{g}}^{(2)}} > 0$, with $H_{{\overline{g}}^{(2)}}=0$, then the functions $\sigma_{i,j}(s)$ are constant, strictly increasing or strictly decreasing, besides that, we prove that the zeros of these functions form two sequences of degeneracy instants (for whose the operator $\mathcal{J}_s$ is singular), $(s_{n}^{(1)})_{n} \rightarrow 0$ and $(s_{n}^{(2)})_{n} \rightarrow \infty$; anywhere else $\mathcal{J}_s$ is an isomorphism (and the family is rigid). Except for a finite number of these degeneracy instants, $s_{1}^{(2)} < s < s_{1}^{(1)}$, we prove that there is a jump in the Morse index of the operator $\mathcal{J}_s$ as it varies from $s -\varepsilon$ to $s +\varepsilon$, for $\varepsilon > 0$, where $\mathcal{J}_{s -\varepsilon}$ and $\mathcal{J}_{s +\varepsilon}$ are non-singular. In the case of \emph{neutral instants} the same instant $s_{1}^{(2)} < s < s_{1}^{(1)}$ can be a zero of a strictly increasing eigenvalue and of a strictly decreasing eigenvalue, that is,  the dimensions of the negative eigenspaces, $V_{s-\varepsilon}^-$ and $V_{s+\varepsilon}^-$, may be equal, making impossible to ensure that a jump in the Morse index occurs. In this paper we use a more accurate analysis, considering an action of a Lie group $G$ on the product manifold $M$ and examining the non-equivalence of the representations of $G$ in these eigenspaces.

The construction discussed at Section \ref{sec:iso_rep}, allows us to present the following definition.

\begin{definition}\label{def:harm}
	Let $G$ be a Lie group isometric action on a Riemannian manifold $(N,h)$. We say that the action of $G$ is harmonically free if, given an arbitrary family of eigenspaces of the Laplacian $\Delta_h$, pairwise distinct,
	$$V_1, V_2, \ldots, V_r, V_{r+1}, \ldots, V_{r+s}, \ \ \ r, s \geq 1,$$
	and integers $n_{\ell} \geq 0$, with $\ell = 1, \ldots, r+s$, not all zero, then the (anti-)representations
	$$\bigoplus_{\ell = 1}^{r}  n_{\ell} \cdot \pi_{\ell} \ \ \mbox{ e }  \ \ \bigoplus_{\ell=r+1}^{r+s} n_{\ell} \cdot \pi_{\ell},$$
	are non-equivalents. Here, we denote by $\pi_{\ell}$ the (anti-)representation of $G$ in $V_{\ell}$, defined in Proposition \ref{prop:repres}, and $n_{\ell} \cdot\pi_{\ell}$ denotes the (external) direct sum of $n_\ell$ copies of $\pi_{\ell}$.
\end{definition}

Now we present some concrete examples of this action.

\begin{example}
	The sphere $\mathbb{S}^n$ embedded in $\mathbb{R}^{n+1}$. The isometry group of $\mathbb{R}^{n+1}$ is the orthogonal group $O(n+1)$. Consider $\mathbb{S}^n$ with the induced euclidean metric in $\mathbb{R}^{n+1}$, then, the subgroup $SO(n+1)$, of $O(n+1)$, acts in $\mathbb{S}^n$ by isometries; on the other hand, any isometry of $\mathbb{S}^n$ can be extended to an isometry of $\mathbb{R}^{n+1}$, in a way that $O(n+1)$ is the isometry group of the sphere of dimension $n$. The action of $O(n+1)$ in the sphere is transitive (which implies that $\mathbb{S}^n$ is homogeneous). 
	We can consider the action of $G_0=SO(n+1)$ on $\mathbb{S}^n$, which is transitive too. Now, the isotropy subgroup of a point $p \in \mathbb{S}^{n}$, denoted by $G_0(p)$, is isomorphic to $SO(n) \times \{1\}$. Indeed, note that the isotropy subgroup of $e_{n+1}$ is given by
	$$
	G_0(e_{n+1}) = \left\{ 
	\overline{A} = \left( 
	\begin{array}{cc} 
	A & 0 \\ 
	O & 1 
	\end{array} 
	\right) : A \in SO(n), O \mbox{ é a matriz nula } n \times n \right\},
	$$
	which is isomorphic to $SO(n) \times \{1\}$. 
	Then, $\mathbb{S}^n$ is isotropic for all $p$, besides that, $\mathbb{S}^{n}$ is diffeomorphic to $SO(n+1)/SO(n)$. The representations of $SO(n+1)$ in the eigenspaces of the Laplacian of $\mathbb{S}^n$ are irreducible (this is a consequence of a well-known result that we present below as the Proposition \ref{prop:Berger}). The eigenvalues of the Laplacian of $\mathbb{S}^n$, $n \geq 2$, are known e given by $\lambda = k(k+n-1)$, with $k=0,1,2,...$ and the dimensions of the correspondents eigenspaces, given by $\mbox{dim}V_{\lambda} = \binom{n+k}{n} - \binom{n+(k-2)}{n}$, form an strictly increasing sequence \cite[Cap.\ III, \S 22]{Shubin}, so the eigenspaces are non-equivalent\footnote{When $n=1$, it is known that the representations of the Laplacian of $\mathbb{S}^1$ in the eigenspaces of $\Delta$ are pairwise equivalent.}. So, the action of $SO(n+1)$ on $\mathbb{S}^n$ is an example of harmonically free action.
\end{example}

We can generalize the sphere case in the next example.
\begin{example}
	Let $(N,h)$ be a compact Riemannian manifold. When the eigenspaces of the Laplacian $\Delta_h$ are irreducible and pairwise non-equivalent, the action of the isometry group $Iso(N,h)$ in $N$ is harmonically free.
\end{example}

\begin{example} 
	The next result (\cite[Proposition C.I.8]{Berger}), allows us to exhibit an important class of examples of harmonically free actions. Let $G$ be a compact (connected) Lie group, $H$ be a closed subgroup of $G$ and $M = G/H$ be the correspondent homogeneous space with a Riemannian metric $g$, which is $G$-invariant. Let $\Delta_g$ be the Laplacian of $g$ in $M$ and denote by $V_{\lambda}$, the eigenspaces of $\Delta_g$ associated to the eigenvalues $\lambda$. We know that $V_{\lambda}$ have finite dimensions, for all $\lambda$ and are invariants by the action of $G$. Consider the connected component of identity in $H$, denoted by $H_0$. Remember that $H$ acts in $T_pM$, where $p$ is the base point of $G/H$, via isotropy linear representation
	\begin{align*}
	Is_p :
	& H \longrightarrow GL(T_pM) \\
	& h \longmapsto (\mbox{d}h)_p.
	\end{align*} 
	\begin{proposition}\label{prop:Berger}
		If the action of $H_0$ is transitive on the unitary sphere of $T_pM$, then the representations
		$$\rho_{\lambda} : G \longrightarrow GL(V_{\lambda}),$$
		are irreducible, for all $\lambda$.
	\end{proposition}
	Concretely talking, this result applies to compact symmetric spaces of rank $1$.	
	\begin{corol}
		The natural action of the isometry group of a compact symmetric space of rank one is harmonically free.
	\end{corol}
	Moreover, in these cases, the dimensions of the eigenspaces of the Laplacian of $M$, with respect to a $G$-invariant metric, forms a strictly increasing sequence, which implies that the eigenspaces are pairwise non-equivalent.
\end{example}

\section{Equivariant bifurcation of solutions of the Yamabe problem}\label{sec:main}

Now we are ready to present our main result. Let $(M_1, g^{(1)})$ be a compact Riemannian manifold with positive constant scalar curvature and let $(M_2, \overline{g}^{(2)})$ be a compact Riemannian manifold with boundary $\partial M \neq \emptyset$, positive constant scalar curvature and vanishing mean curvature of the boundary. Assume that the pair $(g^{(1)}, \overline{g}^{(2)})$ is non-degenerate. Consider the product manifold $M = M^{(1)} \times M^{(2)}$, of dimension $m \geq 3$ and boundary $\partial M = M^{(1)} \times \partial M^{(2)} \neq \emptyset$. Let $\overline{g}_s = g^{(1)} \oplus s \overline{g}^{(2)}$, with $s>0$, be a family of metrics in $M$. In fact, $\{\overline{g}_s\}_{s>0}$ is a family of metrics of constant scalar curvature and vanishing mean curvature of the boundary $\partial M$.  Let $G$ be a connected Lie group action by isometries on $M_1$.

\begin{theorem}\label{thm:maineq}
 If the action of $G$ on $M_1$ is harmonically free, then every degenerate instant for the family of operators $\left\{\mathcal{J}_s \right\}_{s>0}$ is an instant of bifurcation for the family $\{\overline{g}_s\}_{s>0}$ of solutions of the Yamabe problem in the compact manifold $M$ with minimal boundary.
\end{theorem}
\begin{proof}
	Without loss of generality, suppose that $G$ acts in $M_1$ and define a non-trivial action of $G$ on $M$ as
	$$
	\begin{array}{rcl}
	G \times M & \longrightarrow & M \\
	(\xi, (x_1, x_2)) & \longmapsto & (\xi \cdot x_1, x_2),
	\end{array}
	$$
	where $x_1 \in M_1$ and $x_2 \in M_2$, that is, the action of $G$ on $M_2$ is trivial.
	
	Remember that there is a (symmetric) Laplacian operator,  $\Delta_{g^{(1)}}$, associated to the manifold $M_1$, which has eigenvalues 
	$$0 = \rho^{(1)}_0 < \rho^{(1)}_1 < \rho^{(1)}_2 < \ldots$$
	with finite multiplicities $\mu^{(1)}_i$, with $i = 0, 1, 2, \ldots$, to each of then there is a correspondent eigenspace $V^{(1)}_i$, of finite dimension equals to $\mu^{(1)}_i$, for all $i \geq 0$. As $G$ acts by isometries on $M_1$, we can see the eigenspace of $\rho^{(1)}_i$ as a $G$-space, so $G$ has a natural (anti-)representation in $V^{(1)}_i$,
	$$\pi^{(1)}_i : G \longrightarrow GL(V^{(1)}_i),$$
	given by $\pi^{(1)}_i (\xi)f = f \circ \xi$, $\forall f \in V^{(1)}_i$. 
	
	For the manifold $M_2$,
	$$0 = \rho_{0}^{(2)} < \rho_{1}^{(2)} < \rho_{2}^{(2)} < \ldots $$
	is the sequence of all distinct eigenvalues of the operator $\Delta_{{\overline{g}}^{(2)}}$, subjected to the Neumann boundary condition,
	\begin{equation}
		\left\{
		\begin{array}{rcl}
			\Delta_{{\overline{g}}^{(2)}} f^{(2)} & = & \rho_{j}^{(2)} f^{(2)},\mbox{ on $M^{(2)}$,} \\
			\partial_{\eta_2} f^{(2)} & = & 0, \mbox{ on $\partial M^{(2)}$},
		\end{array}
		\right.
	\end{equation}
	to each of them there is a correspondent eigenspace $V^{(2)}_j$ of finite dimension equals to $\mu^{(2)}_j$, which denotes the geometric multiplicity of $\rho_{j}^{(2)}$, for each $j  \geq 0$. As the action of $G$ on $M_2$ is trivial, the representation of $G$ in the eigenspace $V^{(2)}_j$ is the trivial representation, that is,
	$$\pi^{(2)}_j : G \longrightarrow GL(V^{(2)}_j),$$
	is given by $\pi^{(2)}_j (\xi)f = f \circ id = f$, $\forall f \in V^{(2)}_j$.
	
	In the product manifold $M$, the Laplacian is given by
	$$\Delta_{\overline{g}_s} = \Delta_{g^{(1)}} \otimes I + \displaystyle\frac{1}{s} \left( I \otimes \Delta_{\overline{g}^{(2)}}\right)$$
	and has eigenvalues $\rho_{i,j} =  \rho^{(1)}_i + \displaystyle\frac{1}{s} \rho^{(2)}_j$, with multiplicities $\mu_{i,j} = \mu_i^{(1)} \cdot \mu_j^{(2)}$, which are the dimensions of the respective eigenspaces $V_{i,j} = V_i^{(1)} \otimes V_j^{(2)}$, for all $i,j \geq 0$. If $\Sigma(\Delta_{\overline{g}_s})$ is the spectrum of the Laplacian in the product manifold $(M, \overline{g}_s)$, the spectrum of the operator $\mathcal{J}_s = \Delta_{\overline{g}_s} - \displaystyle\frac{R_{\overline{g}_s}}{m-1}$ is given by
	$$\Sigma(\mathcal{J}_s) = \left\{ \lambda - \displaystyle\frac{R_{\overline{g}_s}}{m-1} : \lambda \in \Sigma(\Delta_{\overline{g}_s}) \backslash \{0\} \right\};$$
	explicitly
	$$\Sigma(\mathcal{J}_s) = \left\{ \sigma_{i,j}(s) : i,j \geq 0, i + j >0 \right\},$$
	where the eigenvalues are given by
	$$\sigma_{i,j}(s) = \left( \rho_i^{(1)} - \displaystyle\frac{R_{g^{(1)}}}{m-1} \right) + \displaystyle\frac{1}{s} \left( \rho_j^{(2)} - \displaystyle\frac{R_{\overline{g}^{(2)}}}{m-1} \right),$$
	and satisfy the Neumann problem in the product manifold with boundary
	\begin{equation}
		\left\{
		\begin{array}{rcl}
			\mathcal{J}_s f & = & \sigma_{i,j} f ,\mbox{ on $ M$,} \\
			\partial_{\eta_s} f  & = & 0, \mbox{ on $\partial M$},
		\end{array}
		\right.
	\end{equation}
	where the function $f = f^{(1)} \otimes f^{(2)}$ is an eigenfunction of $\mathcal{J}_s$ associated to the eigenvalue $\sigma_{i,j}$ if and only if it is eigenfunction of $\Delta_{\overline{g}_s}$ associated to the eigenvalue $\rho_{i,j} \neq 0$, hence $\sigma_{i,j}$ has geometric multiplicity equal to the product $\mu_{i}^{(1)} \cdot \mu_{j}^{(2)}$. 
	
	Now, for each eigenvalue $\rho_{i,j} \neq 0$ of $\Delta_{\overline{g}_s}$ (or, equivalently, for each eigenvalue $\sigma_{i,j}$ of $\mathcal{J}_s$), there exists a representation of $G$ in the eigenspace $V_{i,j} = V_i^{(1)} \otimes V_j^{(2)}$ given by
	$$
	\begin{array}{rcl}
	\pi_{i,j} : G & \longrightarrow & GL(V_{i,j}) \\
	\xi & \longmapsto & \pi_{i,j}(\xi)
	\end{array},
	$$
	where
	$$\pi_{i,j}(\xi)f = \pi_i^{(1)}(\xi) \otimes \pi_j^{(2)}(\xi) f_1 \otimes f_2 = \pi_i^{(1)}(\xi) f_1 \otimes \pi_j^{(2)}(\xi)f_2 = (f_1 \circ \xi) \otimes (f_2 \circ id) = (f_1 \circ \xi) \otimes f_2,$$ 
	for each $g \in G$ e $f \in V_{i,j}$ . 
	
	Let $s_*$ be a degeneracy instant for the family of metrics $\{\overline{g}_s\}_s$, in particular, $s_*$ could be a \emph{neutral instant}. Let $\varepsilon > 0$ be sufficiently small so that $s_*$ be the only degeneracy instant of $\mathcal{J}_{( \cdot )}$ in the interval $[s_*-\varepsilon, s_*+\varepsilon]$, distinctly $\mathcal{J}_{( \cdot )}$ is non-singular at the extremes of this interval. The representations of $G$ on the negative eigenspaces of $\mathcal{J}_{s_* - \varepsilon}$ e $\mathcal{J}_{s_* + \varepsilon}$, denoted by
	$$ \pi_{s_* - \varepsilon}^- : G \longrightarrow GL(V_{s_*-\varepsilon}^-) \ \ \ \ \mbox{e} \ \ \ \  \pi_{s_* + \varepsilon}^- : G \longrightarrow GL(V_{s_*+\varepsilon}^-) ,$$
	respectively, are given by the direct sums of the representations of $G$ on the eigenspaces $V_{i,j}$ that are related to the negative eigenvalues $\sigma_{i,j}( \ \cdot \ )$, of the operator $\mathcal{J}_{(\ \cdot \ )}$ at instants $s_*-\varepsilon$ and $s_*+\varepsilon$. Note that the number of negative eigenvalues and the dimensions of the associated eigenspaces are the same for the operators $J_{s_* - \varepsilon}$ and $J_{s_* + \varepsilon}$, \emph{except} for those that have $s_*$ as a zero. Denote by $V_0$ the direct sum of the eigenspaces associated to eigenvalues that take negative values for all  $s \in [s_*-\varepsilon, s_*+\varepsilon]$.
	
	Let $\sigma_{i,j}$ be an eigenvalue that vanishes at $s_*$. If $\sigma_{i,j}$ is increasing, the eigenspace $V_{i,j}$ takes part in the negative eigenspace of the operator $\mathcal{J}_s$, for all $s<s_*$, but as $s$ varies, passing trough $s_*$, i.e. when $s>s_*$, $V_{i,j}$ is no longer part of the negative eigenspace of $\mathcal{J}_s$. The contrary occurs if $\sigma_{i,j}$ is decreasing; in this case, $V_{i,j}$ is added to the negative eigenspace of $\mathcal{J}_s$, only when $s>s_*$. Thus, we can written the negative eigenspace of the operator $\mathcal{J}_{s_*-\varepsilon}$ as
	$$V_{s_*-\varepsilon}^- = V_0 \oplus \left( \bigoplus_{k=1}^{r} \left(V_{i_k}^{(1)} \otimes V_{j_k}^{(2)}\right) \right),$$
	where $V_{i_k,j_k} = V_{i_k}^{(1)} \otimes V_{j_k}^{(2)}$ is the eigenspaces associated to increasing eigenvalues, such that $\sigma_{i_k, j_k}(s_*) = 0$. Similarly, the negative eigenspace of $\mathcal{J}_{s_*+\varepsilon}$ can be written as
	$$V_{s_*+\varepsilon}^- = V_0 \oplus \left( \bigoplus_{k=r+1}^{r+s} \left(V_{i_k}^{(1)} \otimes V_{j_k}^{(2)}\right) \right),$$
	where $V_{i_k,j_k} = V_{i_k}^{(1)} \otimes V_{j_k}^{(2)}$ is the eigenspace associated to the decreasing eigenvalues, such that $\sigma_{i_k, j_k}(s_*) = 0$.
	
	Note that $\left\{ V_{i_k}^{(1)}\right\}_{k=1}^{r+s}$ and  $\left\{ V_{j_k}^{(2)}\right\}_{k=1}^{r+s}$ are families of eigenspaces of $\Delta_{g^{(1)}}$ and $\Delta_{\overline{g}^{(2)}}$, respectively, pairwise distinct. The pairs $(i_k, j_k)$ are all different, more than that, we have 
	$$(i_p,j_p) \neq (i_q,j_q) \mbox{  e  } \sigma_{i_p,j_p}(s_*) = \sigma_{i_q,j_q}(s_*) = 0 \Longrightarrow i_p \neq i_q \mbox{  e  } j_p \neq j_q;$$
	indeed, if $i_p = i_q$, as the sequence $j \mapsto \rho_j^{(2)}$ is strictly increasing, there is a unique $j_p$ such that $\sigma_{i_p,j_p}(s_*) = 0$, hence $j_p = j_q$. Symmetrically, if $j_p = j_q$, then $i_p = i_q$.
	
	The representation of $G$ on the negative eigenspace $V_{s_*-\varepsilon}^-$ is given by 
	$$\pi_{s_*-\varepsilon}^- = \pi_0 \oplus \left( \bigoplus_{k=1}^{r} \left(\pi_{i_k}^{(1)} \otimes \pi_{j_k}^{(2)}\right) \right),$$
	and the representation of $G$ on the negative eigenspace $V_{s_*+\varepsilon}^-$ is given by 
	$$\pi_{s_*+\varepsilon}^- = \pi_0 \oplus \left( \bigoplus_{k=r+1}^{r+s} \left(\pi_{i_k}^{(1)} \otimes \pi_{j_k}^{(2)}\right) \right),$$
	where $\pi_{j_k}^{(2)}$ is the trivial representation. We know that, if $U$ and $W$ are vector spaces of finite dimensions, then
	$U \cdot {\mbox{dim} W} = \underbrace{U \times U \times \ldots \times U}_{\mbox{dim} W \mbox{ vezes}}$, has dimension ${\mbox{dim} U} \cdot {\mbox{dim} W}$ and is isomorphic to $U \otimes W$, accordingly, 
	$$\bigoplus_{k=1}^{r} V_{i_k}^{(1)} \cdot \mbox{dim}{V_{j_k}^{(2)}}  \simeq \bigoplus_{k=1}^{r} V_{i_k}^{(1)} \otimes V_{j_k}^{(2)}   \ \ \ \mbox{ e } \ \ \ \bigoplus_{k=r+1}^{r+s} V_{i_k}^{(1)} \otimes V_{j_k}^{(2)} \simeq \bigoplus_{k=r+1}^{r+s}  V_{i_k}^{(1)} \cdot \mbox{dim} V_{j_k}^{(2)};$$
	and as the action of $G$ on $M_1$ is harmonically free, it does not exist isomorphism between the spaces
	$$ \bigoplus_{k=1}^{r} V_{i_k}^{(1)} \otimes V_{j_k}^{(2)} \ \ \mbox{ e } \ \  \bigoplus_{k=r+1}^{r+s} V_{i_k}^{(1)} \otimes V_{j_k}^{(2)}.$$
	Hence, $\left( \pi_{s_*-\varepsilon}^-, V_{s_*-\varepsilon}^-\right)$ e $\left( \pi_{s_*+\varepsilon}^-, V_{s_*+\varepsilon}^- \right)$ are non-equivalent and the result follows from the Theorem~\ref{thm:bifeq}.	
\end{proof}

\addcontentsline{toc}{chapter}{\protect\numberline{}{Bibliografia}}

\end{document}